\documentclass[11pt,letterpaper]{amsart}
\usepackage[dvips]{graphicx}
\usepackage{enumerate}
\usepackage{amscd}
\usepackage{amssymb}

\newtheorem*{maintheorem}{Main Theorem}
\newtheorem*{seifertcorollary}{Corollary~\ref{coro:Seifert}}
\newtheorem*{Smaletheorem}{Theorem~\ref{thm:Smale}}

\newtheorem{introtheorem}{Theorem}
\newtheorem{maincorollary}{Corollary}
\newtheorem{theorem}{Theorem}[section]
\newtheorem{cor}[theorem]{Corollary}
\newtheorem{lemma}[theorem]{Lemma}
\newtheorem{prop}[theorem]{Proposition}

\theoremstyle{definition}

\numberwithin{figure}{section}
\numberwithin{equation}{section}

\newcommand{\Diff}{\operatorname{Diff}}
\newcommand{\diff}{\operatorname{diff}}
\newcommand{\Emb}{\operatorname{Emb}}
\newcommand{\emb}{\operatorname{emb}}
\renewcommand{\H}{\mathbb{H}}
\newcommand{\isom}{\operatorname{isom}}
\newcommand{\Isom}{\operatorname{Isom}}
\newcommand{\Mod}{\operatorname{Mod}}
\newcommand{\N}{\mathbb{N}}
\newcommand{\Out}{\operatorname{Out}}
\newcommand{\R}{\mathbb{R}}
\newcommand{\SF}{\operatorname{SF}}
\newcommand{\SO}{\operatorname{SO}}

\newcommand{\longpage}{\enlargethispage{\baselineskip}}
\newcommand{\shortpage}{\enlargethispage{-\baselineskip}}

\title[The Smale Conjecture for Seifert fibered spaces]{The Smale
conjecture for Seifert fibered\\ spaces with hyperbolic base orbifold}

\date{\today}

\author{Darryl McCullough}
\address{Department of Mathematics, University of Oklahoma, Norman, Oklahoma
73019, USA}
\email{dmccullough@math.ou.edu}
\thanks{The first author was supported in part by NSF grant DMS-08082424}

\author{Teruhiko Soma}
\address{Department of Mathematics and Information Sciences,
Tokyo Metropolitan University,
Minami-Ohsawa 1-1, Hachioji, Tokyo 192-0397, Japan}
\email{tsoma@tmu.ac.jp}

\begin{document}

\begin{abstract}
Let $M$ be a closed orientable $3$-manifold admitting an $\H^2\times\R$ or
$\widetilde{\mathrm{SL}_2}(\R)$ geometry, or equivalently a Seifert fibered
space with a hyperbolic base $2$-orbifold. Our main result is that the
connected component of the identity map in the diffeomorphism group
$\Diff(M)$ is either contractible or homotopy equivalent to~$S^1$,
according as the center of $\pi_1(M)$ is trivial or infinite cyclic. Apart
from the remaining case of non-Haken infranilmanifolds, this completes the
homeomorphism classifications of $\Diff(M)$ and of the space of Seifert
fiberings $\SF(M)$ for compact orientable aspherical $3$-manifolds. We also
prove that when $M$ has an $\H^2\times\R$ or
$\widetilde{\mathrm{SL}_2}(\R)$ geometry and the base orbifold has
underlying manifold the $2$-sphere with three cone points, the inclusion
$\Isom(M)\to \Diff(M)$ is a homotopy equivalence.
\end{abstract}

\maketitle

Let $M$ be a smooth closed manifold and $\Diff(M)$ the space of
diffeomorphisms of $M$ with the $C^\infty$-topology. The path component of
$\Diff(M)$ containing the identity $\mathrm{Id}_M$ is denoted by
$\diff(M)$. In this paper, we focus on the case when $M$ is a closed
orientable $3$-manifold admitting an $\H^2\times\R$ or
$\widetilde{\mathrm{SL}_2}(\R)$ geometry, or equivalently $M$ is a Seifert
fibered space with a hyperbolic base $2$-orbifold. Waldhausen~\cite{wa2}
and, for the non-Haken cases, Scott \cite{sc3} together with Boileau-Otal
\cite{bo} proved that for such $M$, an element $f$ of $\mathrm{Diff}(M)$
belongs to $\diff(M)$ if and only if $f$ is homotopic to $\mathrm{Id}_M$,
and consequently homotopic diffeomorphisms are isotopic. A new proof, based
on the insulator methods of Gabai~\cite{ga1}, was given by the second
author in~\cite{so}. Our main result is the following:
\begin{maintheorem}\label{t_main}
Let $M$ be a closed orientable Seifert fibered space with a hyperbolic base
$2$-orbifold. Then $\diff(M)$ is contractible or is homotopy equivalent
to~$S^1$, according as the center of $\pi_1(M)$ is trivial or infinite
cyclic.\par
\end{maintheorem}

As we will see, combined with known results the Main Theorem reduces two
longstanding conjectural pictures in the topology of compact orientable
aspherical $3$-manifolds to a single remaining case, namely that of
non-Haken infranilmanifolds. The first conjectural picture is the
homeomorphism classification of $\Diff(M)$. It is known that $\Diff(M)$ is
an infinite-dimensional separable Fr\'echet manifold, so its homeomorphism
type is determined by its homotopy type. Moreover, since $\Diff(M)$ is a
topological group, any two components are homeomorphic. Therefore the
homeomorphism type of $\Diff(M)$ is determined by the cardinality of the mapping
class group $\Mod(M)$ and the homotopy type of $\diff(M)$.

Here and throughout, we denote by $k=k(M)$ the rank of the center of
$\pi_1(M)$, which is $0$ if $M$ does not admit a Seifert fibering. When $M$
is Seifert-fibered, $k$ is $3$ if $M$ is the $3$-torus, is $1$ when $M$ is
the orientable circle bundle over the Klein bottle that admits a
cross-section, and in all other cases is $1$ or $0$ according as the base
$2$-orbifold of $M$ is orientable or not. By $(S^1)^k$, we mean the product
of $k$ copies of $S^1$, where $(S^1)^0$ means a single point.

From work of Hatcher \cite{ha1} and Ivanov \cite{I1,I2}, we know that for
Haken $3$-manifolds, possibly with nonempty boundary, $\diff(M)\simeq
(S^1)^k$ except in two cases: the solid torus, for which $\diff(M)\simeq
S^1\times S^1$, and $D^3$, for which
$\diff(M)\simeq\SO(3)$~\cite{ha2}. Apart from these exceptional cases, the
path component $\isom(M)$ of $\mathrm{Id}_M$ in the isometry group
$\Isom(M)$ is $(S^1)^k$, when one uses a metric on $M$ of maximal symmetry
(that is, one for which the Lie group $\Isom(M)$ has maximal dimension and
maximal number of components), and the homotopy equivalence $(S^1)^k\to
\diff(M)$ is simply the inclusion $\isom(M)\to \diff(M)$. For the
exceptional Haken cases, $\isom(M)\to \diff(M)$ is still a homotopy
equivalence. For hyperbolic $M$, Haken or not, Gabai~\cite{ga2} proved that
$\diff(M)$ is contractible; in this case $k=0$ and $\isom(M)$ is a point so
$\isom(M)\to \diff(M)$ is again a homotopy equivalence.

Among the closed orientable aspherical $3$-manifolds, there remain only
the non-Haken Seifert fibered cases. It is well-known that such a manifold
must have base orbifold a $2$-sphere with three cone points, and such a
Seifert fibered manifold is non-Haken if and only if its first homology
group is finite~\cite{wa1}. They have $k=1$ and (as we will check) $\isom(M)=
S^1$. There are two classes:
\begin{enumerate}
\item[1.] The non-Haken manifolds among those of the Main Theorem.
\item[2.] The non-Haken infranilmanifolds. A nilmanifold is a $3$-manifold
that is a quotient of Heisenberg space by a torsion-free lattice;
topologically these are the $S^1$-bundles over the torus with nonzero
Euler class. An infranilmanifold is a finite quotient of a
nilmanifold. Their base orbifolds have cone points of types $(2,4,4)$, 
$(2,3,6)$, or $(3,3,3)$.
\end{enumerate}
\shortpage

The homotopy equivalence $S^1\to \diff(M)$ in the Main Theorem is realized
as the inclusion $\isom(M)\to \diff(M)$, when $M$ has its standard
geometry. Therefore, combining the previous results, we have
\begin{introtheorem}
Let $M$ be a compact orientable aspherical $3$-manifold with a metric of
maximal symmetry, other than a non-Haken infranilmanifold. Then the inclusion
$\isom(M)\to \diff(M)$ is a homotopy equivalence.
\label{thm:weakSmale}
\end{introtheorem}
\noindent Since any two infinite-dimensional separable Fr\'echet spaces are
homeomorphic, we have as a corollary to Theorem~\ref{thm:weakSmale} the
homeomorphism classification of $\Diff(M)$ in the compact orientable
aspherical case:
\begin{maincorollary}
Let $M$ be a compact orientable aspherical $3$-manifold, other than a
non-Haken infranilmanifold. Give $M$ a metric of maximal symmetry. Then
$\Diff(M)$ is homeomorphic to $\Mod(M)\times \isom(M)\times F$, where $F$
is an infinite-dimensional separable Fr\'echet space.
\end{maincorollary}

The homotopy equivalence in Theorem~\ref{thm:weakSmale} may be considered
to be a weak form of the original Smale Conjecture, which asserts that
$\Isom(S^3)\to \Diff(S^3)$ is a homotopy equivalence for the round
$3$-sphere. The original Smale Conjecture was proven in two stages by
J. Cerf~\cite{cerf} and A. Hatcher~\cite{ha2}. For Haken $3$-manifolds,
$\Isom(M)\to\Diff(M)$ often fails to be surjective on path components, but
for the ``small'' manifolds among those in the Main Theorem, we will obtain
the strong form of the Smale Conjecture.
\begin{Smaletheorem}
Let $M$ be a closed orientable Seifert-fibered $3$-manifold having an
$\H^2\times\R$ or $\widetilde{\mathrm{SL}_2}(\R)$ geometry, and
base orbifold a $2$-sphere with three cone
points. Then the inclusion $\Isom(M)\to \Diff(M)$ is a homotopy
equivalence.
\end{Smaletheorem}
\noindent The same statement was proven for closed hyperbolic $3$-manifolds
by Gabai \cite{ga2}. It is known for some elliptic $3$-manifolds but not
others, see~\cite{HKMR}.

The second conjectural picture affected by the Main Theorem concerns the
space of Seifert fiberings $\SF(M)$, defined in Section~\ref{sec:Seifert}.
It is also a separable infinite-dimensional Fr\'echet manifold.
For Haken $3$-manifolds, possibly with boundary, Theorem~3.8.2
of~\cite{HKMR} is
\begin{introtheorem}\label{thm:haken_sf}
Let $\Sigma$ be a Seifert-fibered Haken $3$-manifold. Then each component
of $\SF(\Sigma)$ is contractible.\par
\label{thm:HKMR}
\end{introtheorem}
\noindent Problem 3.47(A3) of the Kirby Problem List~\cite{kirby} is the
conjecture that if $M$ has either the $\H^2\times\R$ or
$\widetilde{\mathrm{SL}_2}(\R)$ geometry, then $\SF(M)$ is contractible.
We will prove that in Section~\ref{sec:Seifert}:
\begin{seifertcorollary}
Let $M$ be a closed orientable Seifert-fibered $3$-manifold 
with a hyperbolic base orbifold. Then $\SF(M)$ is contractible.
\end{seifertcorollary}
\noindent Combining this with Theorem~\ref{thm:haken_sf} yields
\begin{maincorollary}\label{coro:SF}
Let $M$ be a compact orientable aspherical Seifert fibered space,
other than a non-Haken infranilmanifold. Then each component of $\SF(M)$ is
contractible.\par
\end{maincorollary}
\noindent Since the Seifert fiberings on compact $3$-manifolds are
completely classified, Corollary~\ref{coro:SF} gives an effective
homeomorphism classification of $\SF(M)$ for almost all compact aspherical
$3$-manifolds:
\begin{maincorollary}
Let $M$ be a compact orientable aspherical Seifert fibered space,
other than a non-Haken infranilmanifold. Then $\SF(M)$ is homeomorphic to
$E\times F$, where $E$ is the discrete set of equivalence classes of
Seifert fiberings, and $F$ is an infinite-dimensional separable Fr\'echet
space.
\end{maincorollary}

The methods of our paper do not adapt to infranilmanifolds, since we rely
heavily on the hyperbolicity of the base orbifold. But we know of no reason
not to expect that all of the previous results that exclude these manifolds
are actually true for them as well. Consequently, as discussed at the
beginning of Section~\ref{sec:threecone}, we have structured the
applications sections in such a way that if the Main Theorem is proven in
the infranilmanifold case, then all the results listed above will be
established in that case as well.

Section~\ref{sec:sketch} will give a brief overview of the proof of the
Main Theorem, while Sections~\ref{Least} through~\ref{Proof} of this paper
will give the details. Section~\ref{sec:Seifert}, preceded by three
sections of background results, gives the proofs of
Corollary~\ref{coro:Seifert} and Theorem~\ref{thm:Smale}.

We are grateful to K. Ohshika for helpful conversations. We also wish to
acknowledge that J. H. Rubinstein has considered a similar approach to the
Main Theorem, and we thank him for many useful discussions of this and
other problems.

\section{Sketch of the proof of the Main Theorem}\label{sec:sketch}

Palais \cite{pa} showed that $\diff(M)$ has the homotopy type of a
CW-complex, so by the Whitehead Theorem, it suffices to show that
$\pi_n(\diff(M))$ is isomorphic to $\pi_n(S^1)$ for all $n\in
\mathbb{N}$. When $M$ is Haken, the Main Theorem follows from work of
Hatcher \cite{ha1,ha2} and Ivanov \cite{I1,I2}. So we may assume that $M$
is non-Haken, in which case the base orbifold is hyperbolic with underlying
space the $2$-sphere and singular locus consisting of three points. Note
that in these cases,~$k(M)=1$.

Our proof of the Main Theorem incorporates many of the ideas of Gabai's
proof of the Smale Conjecture for closed hyperbolic 3-manifolds \cite{ga2}.
His approach draws on his rigidity theorem for hyperbolic 3-manifolds in
\cite{ga1}. In place of the latter, we will use results from \cite{so}, in
which Scott's rigidity theorem for Seifert fibered spaces \cite{sc1} was
obtained as a 2-dimensional (and hence easier) version of Gabai's rigidity
theorem.

The first step, carried out in Sections~\ref{Least} and \ref{FindTori}, is
to consider an arbitrary Riemannian metric $\nu$ on $M$ and show, using
least-area techniques from~\cite{so}, that the preimage $c^\natural$ in $M$
of a fixed cone point in the base orbifold is the core circle of a
\textit{canonical (open) solid torus.} The canonical torus depends only on
$\nu$ and has certain key limiting properties as $\nu$ is varied. Roughly
speaking, the canonical solid tori for a convergent sequence of metrics
converge to an open solid torus that contains the canonical torus for the
limit metric. These properties are developed and used in the proof of
Lemma~\ref{Cell}.

Lemma~\ref{Cell} corresponds to the Coarse Torus Isotopy Theorem of Gabai
\cite[Theorem 4.6]{ga2}. Given a continuous map $f\colon S^n\to \diff(M)$,
its output is a family of solid tori associated to the cells of a cell
decomposition of an $(n+1)$-ball $B^{n+1}$ with boundary $S^n$. These solid
tori satisfy (1)~for $y\in S^n$, $f(y)(c^\natural)$ is a core of each solid
torus associated to a cell that contains $y$, and (2)~they are nested
according to the corresponding nesting of the cells of $B^{n+1}$. The key
idea of the proof is Gabai's: push forward the standard metric of $M$ using
the diffeomorphisms of $f$ to obtain a map from $S^n$ to the contractible
space of Riemannian metrics on $M$, extend this map to $B^{n+1}$, and use
the canonical solid tori associated to these metrics to get started on
constructing the solid tori of the conclusion.

The final part of the proof, in Section~\ref{Proof}, uses the nested solid
tori from Lemma~\ref{Cell} to construct an extension of a representative
$f\colon S^n\to \diff(M)$ of an element of $\pi_n(\diff(M))$ to a map
$F\colon B^{n+1}\to \diff(M)$. Unlike the hyperbolic case, however,
$\diff(M)$ is not simply connected, indeed $\pi_1(\diff(M))\cong
\pi_1(S^1)$ is generated by a circular isotopy that moves points vertically
around the fibers. To handle $\pi_1(\diff(M))$, we utilize a maximal-tree
argument to reduce to the case when each diffeomorphism associated by $f$
to a point of $S^n$ carries $c^\natural$ into a fixed solid torus
neighborhood of $c^\natural$. Under this assumption, $f$ can be seen to be
homotopic to a well-defined element of $\pi_1(\isom(M))$.

\section{Least area annuli with bounded deviation}\label{Least}

\longpage
Throughout the remainder of this paper, all 3-manifolds are assumed to be
orientable.

Let $M$ be a closed Seifert fibered space with the Seifert fibration
$\sigma\colon M\longrightarrow O$ over a hyperbolic triangle orbifold
$O=O(p,q,r)$, where $p,q,r$ are integers with $2\leq p\leq q \leq r$ and
$1/p+1/q+1/r<1$. The cyclic subgroup $\langle\gamma\rangle$ of
$\pi_1(M)$ generated by the element $\gamma$ represented by a regular fiber
of $M$ coincides with the center $Z(\pi_1(M))$ of $\pi_1(M)$.

Let $a\colon F\longrightarrow O$ be an orbifold covering such that $F$ is a
closed hyperbolic surface and $\widehat a\colon\H^2\longrightarrow
F$ the universal covering. Consider the natural quotient epimorphism
$\varphi\colon \pi_1(M)\longrightarrow \pi_1^{\,\mathrm{orb}}(O)
=\pi_1(M)/\langle\gamma\rangle$ and the covering $p\colon X\longrightarrow M$
associated to $\varphi^{-1}(a_*(\pi_1(F)))\subset \pi_1(M)$.  The Seifert
$S^1$-fibration $\sigma$ lifts to an $S^1$-fibration $\sigma_X\colon
X\longrightarrow F$. We have also an $S^1$-fibration
$\widehat{\sigma}\colon\widehat X\longrightarrow \H^2$ and a
covering $\widehat p\colon\widehat X\longrightarrow X$ in the following
commutative diagram.
\[
\begin{CD}
\widehat X @>{\widehat \sigma}>> \H^2\\
@V{\widehat p}VV        @VV{\widehat a}V\\
X @>{\sigma_X}>> F\\
@V{p}VV        @VV{a}V\\
M @>{\sigma}>> O
\end{CD}
\] 

\bigskip

We regard $G:=\pi_1^{\,\mathrm{orb}}(O)$ as an isometric properly
discontinuous transformation group on $\H^2$, and also as the covering
transformation group on $\widehat X$ with respect to $p\circ \widehat p$.
Then, $\widehat{\sigma}$ is $G$-equivariant.

Let $\mathcal{RM}(M)$ be the space of Riemannian metrics on $M$ with
$C^\infty$-topology.  The metrics on $\widehat X$ and $X$ induced from $\nu
\in \mathcal{RM}(M)$ are also denoted by $\nu$.  Since the $\nu$-lengths of
the $S^1$-fibers $\widehat \sigma(x)^{-1}$ $(x\in \H^2)$ are uniformly
bounded, $\widehat{\sigma}$ is a quasi-isometry. In particular, the
boundary $\partial_\infty \widehat X$ of $\widehat X$ as a Gromov
hyperbolic space is naturally identified with $S_\infty^1=\partial \H^2$.

For a closed subset $J$ of $\H^2$, let $\mathcal{N}_d(J,\H^2)$ denote the
closed $d$-neighborhood $\{y\in \H^2\;\vert\;\text{dist}(y,J)\leq d\}$ of
$J$ in $\H^2$.  For any geodesic line $\alpha\in \H^2$,
$A_\alpha^\natural=\widehat \sigma^{-1}(\alpha)$ is an open annulus
properly embedded in $\widehat X$.  For $C>0$, we set $L_C(\alpha)=\widehat
\sigma^{-1}(\mathcal{N}_C(\alpha,\H^2))$, which is a closed neighborhood of
$A_\alpha^\natural$ in $\widehat X$. Note that $L_C(\alpha)$ does not
depend on the Riemannian metric $\nu$ on~$\widehat X$.

A (compact) annulus $A_0$ embedded in $\widehat X$ is
$\nu$-\emph{least area} if $A_0$ has the least area among all immersed
annuli $A_0'$ in $\widehat X$ with $\partial A_0'=\partial A_0$ with
respect to the metric $\nu$ on $\widehat X$.  An open annulus $A$ properly
embedded in $\widehat X$ is said to be a $\nu$-\emph{least area annulus
  associated to} $\alpha$ if $A$ satisfies the following conditions.
\begin{itemize}
\item
There exists $C>0$ with $A\subset L_C(\alpha)$ such that $A$ is properly
homotopic to $A_\alpha^\natural$ in $L_C(\alpha)$.  Here we say that $C$ is
a \emph{deviation} of $A$.
\item
$A$ is $\nu$-\emph{least area}.  This means that any compact
non-contractible annulus in $A$ is a $\nu$-least area annulus in
$\widehat X$.
\end{itemize}

The following lemma is a stronger version of Lemma 2.1 in \cite{so}.

\begin{lemma}\label{l_deviation}
Let $K$ be a non-empty compact subset of $\mathcal{RM}(M)$.  Then there
exists a constant $C_K>0$ such that, for any geodesic line $\alpha$ in
$\H^2$ and any $\nu\in K$, there exists a $\nu$-least area annulus
in $\widehat X$ associated to $\alpha$ with deviation $C_K$.  Moreover,
$C_K$ is a deviation of any $\nu$-least area annulus in $\widehat X$
associated to $\alpha$.
\end{lemma}

\begin{proof}
The base orbifold of $M$ is divided by three geodesic segments
$u_1,u_2,u_3$ into two hyperbolic triangles having interior angles $\pi/p$,
$\pi/q$, $\pi/r$. Since the Fuchsian group $\pi_1(F)$ is residually finite,
we may assume that $a^{-1}(u_1\cup u_2 \cup u_3)$ is a union of simple
closed geodesics $l_1,\dots,l_n$, if necessary replacing $F$ by a suitable
finite covering space.

We will first construct least area annuli associated to geodesic lines that
project to one of the $l_i$. The preimage
$T_i^\natural=\sigma_X^{-1}(l_i)$ is an embedded incompressible torus in
$X$. By Freedman-Hass-Scott \cite{fhs}, there exists an embedded torus
$T_{i,\nu}$ in $X$ which is $\nu$-least area among all tori homotopic to
$T_i^\natural$ in $X$.  Since $K$ is compact, $s_K=\sup_{\nu\in
  K}\{\mathrm{Area}_\nu(T_i^\natural)\}<\infty$. Each component $A_{i,\nu}$
of $\widehat p^{-1}(T_{i,\nu})$ is a $\nu$-least area open annulus
associated to a component of $\widehat a^{-1}(l_i)$.

Next we obtain a uniform derivation $C_K'$ for these least-area annuli.
Since $\mathrm{Area}_\nu(T_{i,\nu})\leq s_K$ for all $\nu\in K$ and $\inf_{\nu\in
  K}\{\inf_{x\in X} \{\mathrm{inj}_\nu(x)\}\}>0$, the
Ascoli-Arzel\`{a} Theorem implies that any sequence
$\{T_{i,\nu_m}\}_{m=1}^\infty$ with $\nu_m\in K$ has a subsequence
converging uniformly to a torus in $X$ homotopic to $T_i$. This shows that
the $A_{i,\nu}$ $(\nu\in K)$ have a common deviation $C_{K,i}'$. We set
$C_K'=\max_i\{C_{K,i}'\}$.

To define $C_K$, consider any geodesic line $\alpha$ in $\H^2$ and let
$\mathcal{L}$ be the set of geodesic lines $\lambda$ in $\H^2$ with
$\widehat a(\lambda)\subset l_1\cup\cdots\cup l_n$. Denote by
$\mathcal{L}^\vee(\alpha)$ the subset of $\mathcal{L}$ consisting of the
$\lambda$ disjoint from $\alpha$. For any $\lambda\in
\mathcal{L}^\vee(\alpha)$, let $e(\lambda)$ be the component of $\H^2
\setminus \mathcal{N}_{C_K'}(\lambda)$ disjoint from $\alpha$.  As was
shown in the proof of \cite[Lemma 2.1]{so}, there exists a constant
$C_K>0$, independent of $\alpha$, with
$\mathcal{N}_{C_K}(\alpha,\H^2)\supset \H^2\setminus
\bigl(\bigcup_{\lambda\in \mathcal{L}^\vee(\alpha)}e(\lambda)\bigr)$.
Figure~\ref{f_lambda} illustrates~$C_K$.
\begin{figure}[hbtp]
\centering
\includegraphics[width=0.45\textwidth]{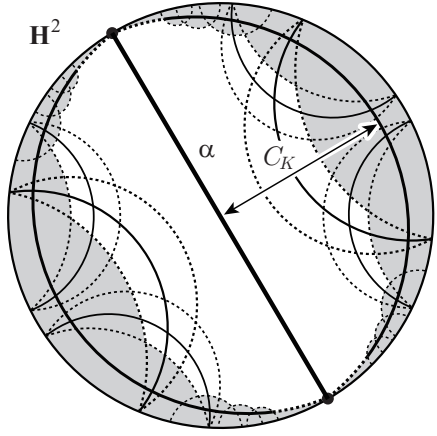}
\caption{The shaded region represents $\bigcup_{\lambda\in
\mathcal{L}^\vee(\alpha)}e(\lambda)$.}
\label{f_lambda}
\end{figure}

We are ready to construct a least-area annulus $A_\alpha$ of
deviation $C_K$ associated to $\alpha$. For any $\lambda\in
\mathcal{L}^\vee (\alpha)$, take a $\nu$-least area annulus $A_\lambda$ in
$\widehat X$ associated to $\lambda$ with deviation $C_K'$. Let
$E(\lambda)$ be the component of $\widehat X\setminus A_\lambda$
quasi-isometric to $e(\lambda)$ via $\widehat \sigma$. Let $\{J_n^+\}$,
$\{J_n^-\}$ be sequences of mutually disjoint $\nu$-least area annuli in
$\widehat X$ associated to elements of $\mathcal{L}\setminus
(\mathcal{L}^\vee(\alpha)\cup \{\alpha\})$ which converge to distinct end
points of $\alpha$ in $\partial_\infty \widehat X= S_\infty^1$ and such
that, for any $n$, the union $J_n^+\cup J_n^-$ excises from $\widehat
X\setminus \bigcup_{\lambda\in \mathcal{L}^\vee(\alpha)} E(\lambda)$ a
solid torus $V_n(\alpha)$ with $V_n(\alpha)\subset V_{n+1}(\alpha)$ and
$\widehat X\setminus \bigcup_{\lambda\in \mathcal{L}^\vee(\alpha)}
E(\lambda)=V_\infty(\alpha)$, where $V_\infty(\alpha)=\bigcup_n
V_n(\alpha)$. Since the boundary of $V_n(\alpha)$ has non-negative mean
curvature, by \cite{fhs} there exists a properly embedded $\nu$-least area
annulus $A_n$ in $V_n(\alpha)$ connecting simple non-contractible loops
$d_n^\pm$ in $J_n^\pm$, as seen in Figure~\ref{f_deviation}.
\begin{figure}[hbtp]
\centering
\includegraphics[width=0.45\textwidth]{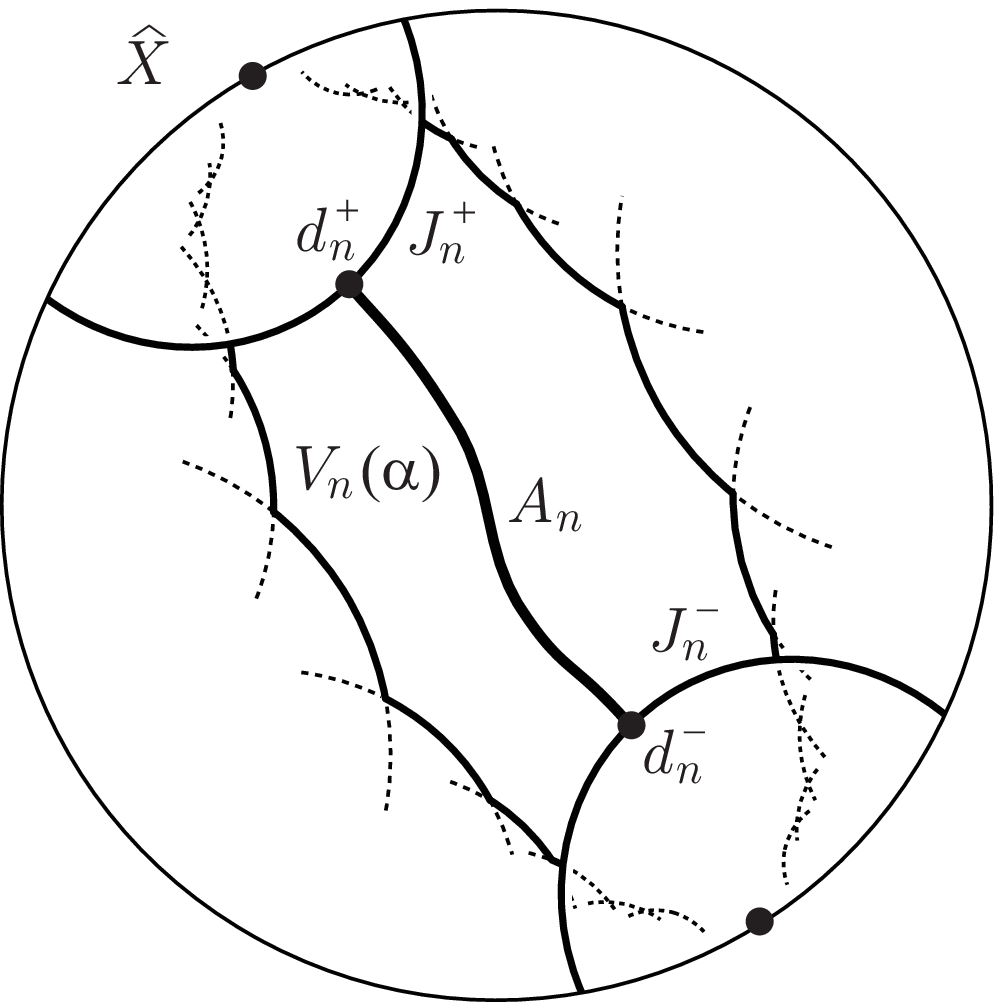}
\caption{}
\label{f_deviation}
\end{figure}
As in the proof of \cite[Lemma 2.1]{so}, one can show that $\{A_n\}$ has a
subsequence converging locally uniformly to a $\nu$-least area annulus
$A_\alpha$ associated to $\alpha$.  Since $A_n\subset V_\infty(\alpha)$, we
have $A_\alpha\subset V_\infty(\alpha)\subset L_{C_K}(\alpha)$. In
particular, $C_K$ is a deviation of~$A_\alpha$.

Now let $A'$ be any $\nu$-least area annulus associated to $\alpha$. For
any $n$, let $\lambda_1^{(n)},\dots,\lambda_{k}^{(n)}$ be the elements of
$\mathcal{L}^\vee(\alpha)$ such that $A_{\lambda_i^{(n)}}$ meets
$V_n(\alpha)$ non-trivially. Choose $m\in \N$ with $m>n$ so that $J_m^+\cup
J_m^-$ is disjoint from $A_{\lambda_1^{(n)}}\cup\cdots\cup
A_{\lambda_k^{(n)}}$. For $\tau\in \{+,-\}$, $A'$ contains a
non-contractible simple loop $l^\tau$ contained in the component of
$\widehat X\setminus J_m^\tau$ disjoint from
$A_{\lambda_1^{(n)}}\cup\cdots\cup A_{\lambda_k^{(n)}}$. Since the
sub-annulus $A_0'$ of $A'$ with $\partial A_0'=l^+\cup l^-$ is $\nu$-least
area, $A_0'\cap (A_{\lambda_1^{(n)}}\cup\cdots\cup
A_{\lambda_k^{(n)}})=\emptyset$.  This shows that $A_n'=A_0'\cap
V_n(\alpha)$ is an annulus properly embedded in $V_n(\alpha)$ and
connecting non-contractible simple loops in $J_n^+$ and $J_n^-$.  Since
$A'=\bigcup_n A_n'$, $A'$ is contained in $V_\infty(\alpha)\subset
L_{C_K}(\alpha)$. We conclude that $C_K$ is a common deviation for all
$\nu$-least area annuli associated to~$\alpha$.
\end{proof}

\begin{lemma}\label{l_outer}
For any $\nu\in K$ and any geodesic line $\alpha$ in $\H^2$,
let $\mathcal{A}_\nu(\alpha)$ be the set of all $\nu$-least area
annuli in $\widehat X$ associated to $\alpha$. Then one of the following
alternatives holds.
\begin{enumerate}[\rm (i)]
\item
$\mathcal{A}_\nu(\alpha)$ consists of a single element
$A_{\alpha[0]}^{\mathrm{out}}\,(\,= A_{\alpha[1]}^{\mathrm{out}})$.
\item
$\mathcal{A}_\nu(\alpha)$ contains two elements
$A_{\alpha[0]}^{\mathrm{out}}$, $A_{\alpha[1]}^{\mathrm{out}}$ with
$A_{\alpha[0]}^{\mathrm{out}}\cap A_{\alpha[1]}^{\mathrm{out}}=
\emptyset$ such that any other elements $A$ of $\mathcal{A}_\nu(\alpha)$
lie between $A_{\alpha[0]}^{\mathrm{out}}$ and
$A_{\alpha[1]}^{\mathrm{out}}$, that is, $A$ is contained in the
component $U$ of $\widehat X\setminus A_{\alpha[0]}^{\mathrm{out}}\cup
A_{\alpha[1]}^{\mathrm{out}}$ with $U\subset L_{C_K}(\alpha)$.
\end{enumerate}
\end{lemma}

The open annuli $A_{\alpha[k]}^{\mathrm{out}}$ given in Lemma \ref{l_outer}
are called the \emph{outermost elements} of $\mathcal{A}_\nu(\alpha)$.

\begin{proof}
We continue to use the notation of Lemma \ref{l_deviation}. In particular,
there is a region $V_\infty(\alpha)\subset L_{C_K}(\alpha)$ for which any
$A\in A_\nu(\alpha)$ is contained in $V_\infty(\alpha)$, and for any
$n\in \N$, $A\cap V_n(\alpha)$ is a $\nu$-least area annulus bounding
non-contractible simple loops in $J_n^+$ and $J_n^-$. 

The closure $\partial_0 V_n(\alpha)$ of $\partial V_n(\alpha)\setminus
(J_n^+\cup J_n^-)$ in $\widehat X$ consists of two annuli. We claim that
some neighborhood of these annuli is disjoint from $\bigcup
\mathcal{A}_\nu(\alpha)$. If not, then there would exist a sequence
$\{A_m\}$ in $\mathcal{A}_\nu(\alpha)$ converging to an element $A_\infty$
in $\mathcal{A}_\nu(\alpha)$ with $A_\infty\cap \partial_0 V_n(\alpha)\neq
\emptyset$. Then, some $A_{\lambda_i^{(n)}}$ given in the proof of Lemma
\ref{l_deviation} and $A_\infty$ would have a tangent point but no
transverse points. A fundamental fact in minimal surface theory implies
that $A_\infty=A_{\lambda_i^{(n)}}$. This contradicts the fact that
$A_\infty\subset V_\infty(\alpha)$, establishing the claim.

By the claim, there exist sub-annuli
$Q_n^\tau$ of $V_n(\alpha)\cap J_n^\tau$ for $\tau\in\{+,-\}$ such that
$\mathrm{Int} Q_n^\tau$ contains $\bigl(\bigcup
\mathcal{A}_\nu(\alpha)\bigr)\cap J^\tau$. We then have mutually disjoint
$\nu$-least area annuli $A_{n,0}$ and $A_{n,1}$ in $V_n(\alpha)$ with
$\partial A_{n,0}\cup \partial A_{n,1}=\partial Q_n^+\cup \partial Q_n^-$
such that the union $A_{n,0}\cup A_{n,1}\cup Q_n^+\cup Q_n^-$ bounds a
solid torus $W_n$ in $V_n(\alpha)$ with $(\bigcup
\mathcal{A}_\nu(\alpha)\bigr)\cap V_n(\alpha)\subset W_n\setminus
(A_{n,0}\cup A_{n,1})$. Passing if necessary to subsequences, we may
assume that both $\{A_{n,0}\}$ and $\{A_{n,1}\}$ converge locally uniformly
to elements $A_{\alpha[0]}^{\mathrm{out}}$,
$A_{\alpha[1]}^{\mathrm{out}}\in \mathcal{A}_\nu(\alpha)$ respectively.
Since $A_{n,0}\cap A_{n,1}=\emptyset$ for all $n\in \N$, if
$A_{\alpha[0]}^{\mathrm{out}} \cap A_{\alpha[1]}^{\mathrm{out}}\neq
\emptyset$, then any elements of the intersection are tangent points but
not transverse points.  This implies that
$A_{\alpha[0]}^{\mathrm{out}}=A_{\alpha[1]}^{\mathrm{out}}$ and hence
$\mathcal{A}_\nu(\alpha)$ is the single element set
$\{A_{\alpha[0]}^{\mathrm{out}}\}$.  In the case of
$A_{\alpha[0]}^{\mathrm{out}}\cap A_{\alpha[1]}^{\mathrm{out}}= \emptyset$,
since $\bigl(\bigcup \mathcal{A}_\nu(\alpha)\bigr)\cap V_n(\alpha)\subset
W_n\setminus (A_{n,0} \cup A_{n,1})$ for any $n\in \N$, any elements of
$\mathcal{A}_\nu(\alpha)\setminus \{A_{\alpha[0]}^{\mathrm{out}},
A_{\alpha[1]}^{\mathrm{out}}\}$ lie between $A_{\alpha[0]}^{\mathrm{out}}$
and $A_{\alpha[1]}^{\mathrm{out}}$ in $\widehat X$.
\end{proof}

\section{Canonical solid tori}\label{FindTori}

For any geodesic line $\alpha$ in $\H^2$ and $\nu\in \mathcal{RM}(M)$, let
$A_{\alpha[0]}^{\mathrm{out}}$, $A_{\alpha[1]}^{\mathrm{out}}$ be the
outermost annuli in $\mathcal{A}_\nu(\alpha)$. In this section we will use
these annuli to construct solid tori in $M$. These tori are canonical in
that they depend only on the choice of Riemannian metric~$\nu$.

In the base orbifold $O=O(p,q,r)$, where $2\leq p\leq q \leq r$, fix once
and for all a singular point $\overline x_0$ that corresponds to the fixed
point of an elliptic element of $G=\pi_1^{\,\mathrm{orb}}(O)$ of order
$r$. Fix $x_0\in (a\circ \widehat a)^{-1}(\overline x_0)$ and write the
orbit $Gx_0$ as $\{x_i\}_{i\in \Gamma}$, where $\Gamma$ is an index set
containing $0$. For any $i,j\in \Gamma$ with $i\neq j$, let
$\alpha_{i:j}=\alpha_{j:i}$ denote the perpendicular bisector line of the
geodesic segment in $\H^2$ connecting $x_i$ with $x_j$. For $\ell=0,1$, we
write $A_{i:j[\ell]}^{\mathrm{out}}$ for~$A_{\alpha_{i:j}[\ell]}^{\mathrm{out}}$.

Let $H_{i\prec j[k]}$ be the component of $\widehat X\setminus
A_{i:j[k]}^{\mathrm{out}}$ quasi-isometric to the component of
$\H^2\setminus \alpha_{i:j}$ containing $x_i$ via $\widehat
\sigma$.  If $H_{i\prec j[0]}\subset H_{i\prec j[1]}$, then we set
$H_{i\prec j}^{\mathrm{inn}}= H_{i\prec j[0]}$.  Otherwise set $H_{i\prec
j}^{\mathrm{inn}}=H_{i\prec j[1]}$.  In particular, our definition
implies that $H_{i\prec j}^{\mathrm{inn}}\cap H_{j\prec
i}^{\mathrm{inn}}=\emptyset$.

A simple loop $c$ in an open solid torus $U$ is a \emph{core} if
$U\setminus c$ is homeomorphic to $(D^\circ\setminus \{\mathbf{0}\})\times
S^1$, where $D^\circ$ is the open unit disk in $\R^2$ centered at
the origin $\mathbf{0}$.  A \emph{core} of a solid torus $V$ is a core of
$\mathrm{Int} V$.

As in the proof of \cite[Lemma 3.1]{so}, one can show that, for any $\nu\in
\mathcal{RM}(M)$ and any $i\in \Gamma$, just one component of the
intersection $\bigcap_{j\in \Gamma\setminus \{i\}} H_{i\prec
j}^{\mathrm{inn}}$ is an open solid torus $\widehat U_{i,\nu}$ such that
a core of $\widehat U_{i,\nu}$ is also a core of $\widehat X$, and all other
components are open 3-balls.

Since $G$ acts on both $\H^2$ and
$\widehat X_\nu$ isometrically, the uniqueness of the outermost annuli
implies that
\begin{equation*}
g(A_{\alpha[0]}^{\mathrm{out}}\cup
A_{\alpha[1]}^{\mathrm{out}})=A_{g(\alpha)[0]}^{\mathrm{out}} \cup
A_{g(\alpha)[1]}^{\mathrm{out}}
\end{equation*}
for any $g\in G$.  Consequently, if $x_i=g(x_0)$ for $g\in G$, $\widehat
U_{i,\nu}=g(\widehat U_{0,\nu})$.  >From our construction of $\widehat
U_{i,\nu}$, we know that the stabilizer $\mathrm{stab}_G (\widehat
U_{i,\nu})$ of $U_{i,\nu}$ in $G$ is isomorphic to the stabilizer
$\mathrm{stab}_G(x_i)$ for the action of $G$ on $\H^2$. Since
$\mathrm{stab}_G(x_i)\cong \mathbf{Z}_r$, $U_\nu=p\circ \widehat p(\widehat
U_{i,\nu})$ is an open solid torus in $M$ and the restriction $q_i\colon
\widehat U_{i,\nu}\longrightarrow U_\nu$ of $p\circ \widehat p_i$ on
$\widehat U_{i,\nu}$ is an $r$-fold cyclic covering.  This $U_\nu$ is
called the \textit{$\nu$-canonical solid torus.}

Since $M$ is a Seifert fibered space with hyperbolic base orbifold, there
exists a metric on $M$ modeled on either $\H^2\times \R$ or
$\widetilde{\mathrm{SL}_2}(\R)$, see \cite{th,sc2} for details.  Fix such a
metric, which we will call the \emph{base metric} on $M$ and denote by
$\nu^\natural$. 

We show that, for any geodesic $\alpha$ in $\mathbb{H}^2$,
$A_\alpha^\natural=\widehat{\sigma}^{-1}(\alpha)$ is the unique
$\nu^\natural$-least area annulus associated to $\alpha$.  For suppose that
$A$ is any $\nu^\natural$-least area annulus associated to $\alpha$.  If
$A\neq \widehat{\sigma}^{-1}(\alpha)$, then $\widehat{\sigma}(A)\setminus
\alpha$ would be non-empty.  Hence we have a $\gamma\in
\mathrm{Isom}(\mathbb{H}^2)$ such that $\alpha\cap \gamma(\alpha)
=\emptyset$ but $\widehat{\sigma}(A)\cap \gamma(\widehat{\sigma}(A))$ is a
non-empty compact set.  Then there exists a isometric transformation
$\widehat\gamma$ on $\widehat X_{\nu^\natural}$ covering $\gamma$ such that
$A \cap \widehat\gamma(A)$ is a non-empty compact set.  This contradicts
that both $A$ and $\widehat\gamma(A)$ are $\nu^\natural$-least area, see
for example \cite[Lemma 1.3]{fhs}.  This shows that $A=A_\alpha^\natural$.

Since $A_{i,j}^\natural=\widehat{\sigma}^{-1}(\alpha_{i,j})$ is the unique
$\nu^\natural$-least area annulus associated to $\alpha_{i,j}$, we have
$A_{i,j[0]}^{\mathrm{out}} = A_{i,j[1]}^{\mathrm{out}}$ in
$\widehat{X}_{\nu^\natural}$. Therefore $\widehat
c^\natural=\widehat\sigma^{-1}(x_i)$ is a geodesic core of $\widehat
U_{i,\nu^\natural}$ and $c^\natural=q_i(\widehat {c}\,^\natural)$ is a
geodesic core of $U_{\nu^\natural}$.

\section{Two key lemmas}\label{Lemmas}

\longpage
The two lemmas in this section correspond respectively to the Coarse Torus
Isotopy Theorem and the Local Contractibility Theorem of
Gabai \cite[Theorems 4.6 and 6.3]{ga2}.

To set notation, denote by $B^{n+1}$ the unit $(n+1)$-ball in $\R^{n+1}$
centered at the origin $\mathbf{0}$, and by $S^n=\partial B^{n+1}$ the unit
sphere with base point $y_0=(1,0,\dots,0)\in \R^{n+1}$. We always suppose
that $S^n$ and $B^{n+1}$ have the Riemannian metrics induced from the
standard Euclidean metric on~$\R^{n+1}$.

For any cell-decomposition $\Delta$ of $B^{n+1}$, the set of $i$-cells in
$\Delta$ will be denoted by $\Delta^{(i)}$ and the union $\Delta^{(0)}\cup
\Delta^{(1)}\cup\cdots\cup \Delta^{(i)}$ by $\Delta^{[i]}$.  For a subset
$\Delta_0$ of $\Delta$, $|\Delta_0|:=\bigcup_{\sigma\in \Delta_0}\sigma$ is
the \emph{underlying space} of $\Delta_0$.  For two solid tori $W,V$, the
relation $W\Subset V$ means that $W\subset \mathrm{Int} V$ and $W$ and $V$
have a common core.  Similarly, $c\Subset V$ means that $c$ is a core
of~$V$.

Suppose that $f\colon K\longrightarrow \diff(M)$ is a continuous map. For
$y\in K$, write $f_y$ for the diffeomorphism $f(y)$, and for any $L\subset
K$, write $f_L$ for~$f|_L$.

\begin{lemma}\label{Cell}
Let $f\colon S^n\longrightarrow \diff(M)$ be continuous.  Then there exist
a cell-decomposition $\Delta$ of $B^{n+1}$ and a map $V$ on $\Delta$
satisfying the following conditions.
\begin{enumerate}[\rm (i)]
\item
For any $\sigma\in \Delta$, $V_\sigma:=V(\sigma)$ is a solid torus in $M$
such that if $\kappa$ is a face of $\sigma$, then $V_\kappa \Subset
V_\sigma$.
\item
For any $y\in \sigma\cap S^n$, $f_y(c^\natural)\Subset V_\sigma$.
\end{enumerate}
\end{lemma}
\begin{proof}
Let $\nu_S\colon S^n\longrightarrow \mathcal{RM}(M)$ be the continuous map
defined by the push forward metrics $\nu_S(y)=(f_y)_*(\nu^\natural)$ $(y\in
S^n)$.  Since $\mathcal{RM}(M)$ is contractible, $\nu_S$ extends to a
continuous map $\nu\colon B^{n+1} \longrightarrow \mathcal{RM}(M)$.

We first examine the limiting behavior of canonical solid tori. Suppose
that $\{y_m\}$ is a sequence in $B^{n+1}$. Passing if necessary to a
subsequence, we assume that $\{y_m\}$ converges to a point $y_\infty\in
B^{n+1}$. For any $j\in \Gamma\setminus \{i\}$, let
$A_{i:j,m}^{\mathrm{out}}$ be the outermost $\nu(y_m)$-least area annulus in
$\widehat X$ with $A_{i:j,m}^{\mathrm{out}}= \mathrm{Fr}(H_{i\prec
  j}^{\mathrm{inn}})$.  By Lemma \ref{l_deviation}, again passing if
necessary to a subsequence, we may assume that these annuli
$A_{i:j,m}^{\mathrm{out}}$ converge locally uniformly to
$\nu(y_\infty)$-least area annuli $A_{i:j,\infty}$ in $\widehat X$
associated to $\alpha_{i:j}$, see \cite[Lemma 3.3]{hs}, \cite[Lemma
  3.3]{ga1} and also the proof of \cite[Theorem 0.2]{so}. The
$A_{i:j,\infty}$ may not be outermost $\nu(y_\infty)$-least area
annuli. But as in the proof of \cite[Lemma 3.1]{so}, $\bigcap_{j\in
  \Gamma\setminus \{i\}} H_{i\prec j}$ contains a unique open solid torus
component $\widehat U$ to which the open solid tori $\widehat
U_{i,\nu(y_m)}$ converge locally uniformly as embeddings from the standard
open solid torus $D^\circ\times S^1$, where $H_{i\prec j}$ is the component
of $\widehat X\setminus A_{i:j,\infty}$ containing $H_{i\prec
j}^{\mathrm{inn}}$. Since each $\widehat U_{i,\nu(y_m)}$ is
$G$-equivariant, $\widehat U$ is also $G$-equivariant.  Thus $U=p\circ
\widehat p(\widehat U)$ is an embedded open solid torus in $M$ containing
$U_{\nu(y_\infty)}$.

Now, for any $y\in B^{n+1}$, fix a solid torus $V_{y,n+1}\Subset
U_{\nu(y)}$. For any $y\in S^n$, since $f_y\colon M_{\nu^\natural}\longrightarrow
M_{(f_y)_*(\nu^\natural)}$ is isometric, we may take $V_{y,n+1}$ so that
$f_y(c^\natural)\Subset V_{y,n+1}$.

We claim that there exists $\delta_{y,n+1}>0$ such that $V_{y,n+1}\subset
U_{\nu(z)}$ if $\mathrm{dist}(y,z)<\delta_{y,n+1}$. If not, then we would
have a sequence $\{z_m\}$ in $B^{n+1}$ with $\mathrm{dist}(y,z_m)<1/m$ and
$V_{y,n+1}\not \subset U_{\nu(z_m)}$. Passing if necessary to a
subsequence, we may as above assume that the $U_{\nu(z_m)}$ converge
locally uniformly to an open solid torus $U$ with $U\supset
U_{\nu(y)}$. Since $V_{y,n+1}$ is a compact subset of $U_{\nu(y)}\subset
U$, $V_{y,n+1}$ would be contained in $U_{\nu(z_m)}$ for all sufficiently
large $m$, a contradiction.

Let $B_{n+1}^\circ (y)$ denote the open $\delta_{y,n+1}$-neighborhood of
$y$ in $B^{n+1}$. We choose the $\delta_{y,n+1}$ so that $B_{n+1}^\circ
(y)\cap S^n=\emptyset$ if $y\in \mathrm{Int} B^{n+1}$. Moreover, since
$f_y(c^\natural)$ moves continuously on $y\in S^n$, we may choose the
$\delta_{y,n+1}>0$ so that $f_z(c^\natural)\Subset V_{y,n+1}$ for any $z\in
B_{n+1}^\circ(y)\cap S^n$.

Fix a finite collection $\{B_{n+1}^\circ(y_1),\dots, B_{n+1}^\circ(y_k)\}$
that covers $B^{n+1}$.  Let $\Delta_{n+1}^*$ be a piecewise smooth cell
decomposition on $B^{n+1}$ such that any $(n+1)$-cell $\sigma$ of
$\Delta_{n+1}^*$ is contained in at least one of the $B^\circ(y_i)$. Then,
put $V_\sigma^*=V_{y_i,n+1}$ for some $y_i$ with $B_{n+1}^\circ(y_i)\supset
\sigma$.

Next, we will define a subdivision $\Delta_n^*$ of $\Delta_{n+1}^{*[n]}$.
Let $z$ be any element of $B^{n+1}$.  As above, there exists
$\delta_{z,n}>0$ and a solid torus $V_{z,n}$ satisfying $V_{y_i,n+1}\Subset
V_{z,n}\subset U_{\nu(w)}$ for any $w\in B_n^\circ(z)$ and any $y_i$ $(i\in
\{1,\dots,k\})$ with $z\in B_{n+1}^\circ(y_i)$.  For any element $\tau$ of
$\Delta_{n+1}^{*(n)}$, there exists a finite subset $\{z_1,\dots,z_l\}$ of
$\tau$ such that $\{B_n^\circ(z_1),\dots,B_n^\circ(z_l)\}$ covers $\tau$.
Then we take a cell decomposition $\Delta^*(\tau)$ of $\tau$ such that each
$n$-cell of $\Delta^*(\tau)$ is contained in at least one of the
$B_n^\circ(z_i)$ $(i=1,\dots,l)$.  We set $\Delta_n^*=\bigcup_{\tau\in
  \Delta_{n+1}^{*(n)}}\Delta^*(\tau)$.

If $\sigma\in \Delta^*(\tau)^{(n)}\subset \Delta_n^{*(n)}$, then we set
$V_\sigma^*=V_{z_j,n}$ for some $z_j$ with $B_n^\circ(z_j)\supset \sigma$.
If $\sigma$ is contained in a face of $\sigma'\in \Delta_{n+1}^{*(n+1)}$,
then $\tau$ is the face.  It follows that $V_{\sigma'}^*=V_{y_i,n+1}\Subset
V_{z_j,n}=V_\sigma^*$. 

Repeating this process on descending skeleta, we define cell complexes
$\Delta_{n-1}^*,\dots,\Delta_0^*$ and extend the domain of the function
$V^*$ to $\Delta_{n-1}^{*(n-1)}\cup\cdots\cup \Delta_0^{*(0)}$ so that
$\Delta_i^*$ is a subdivision of $\Delta_{i+1}^{*[i]}$ and
$V^*_{\sigma'}\Subset V_\sigma^*$ whenever $\sigma\in \Delta_i^{*(i)}$ is
in a face of $\sigma'\in \Delta_{i+1}^{*(i+1)}$. The union
\[\Delta^*=\Delta_{n+1}^{*(n+1)}\cup \Delta_n^{*(n)}\cup \cdots\cup
\Delta_0^{*(0)}\] is a cell decomposition on $B^{n+1}$.  

Now form the double $d\Delta^*$ of $\Delta^*$ along $\Delta^*|_{S^n}$,
obtaining a cell decomposition on $dB^{n+1}=S^{n+1}$. Let $(d\Delta^*)^*$
be the dual cell decomposition of $d\Delta^*$. The set $\Delta$ of all
non-empty $\sigma\cap B^{n+1}$ and $\sigma\cap S^n$ for $\sigma \in
(d\Delta^*)^*$ defines a cell decomposition on $B^{n+1}$. We define the
map $V$ satisfying conditions (i) and (ii) of this lemma as follows:
\begin{itemize}
\item
If $\sigma\cap S^n=\emptyset$, then $V_\sigma=V^*_\tau$ for $\tau\in
\Delta^{*(n+1-i)}$ dual to $\sigma$.
\item
If $\sigma\cap S^n\neq \emptyset$ and $\sigma\not\subset S^n$,
$V_\sigma=V^*_\tau$ for $\tau\in \Delta^{*(n+1-i)}$ dual to the double
$d\sigma$ of $\sigma$.
\item
If $\sigma\subset S^n$, then $V_\sigma$ is a solid torus in $\mathrm{Int}
V_{\sigma'}$ obtained by slightly shrinking $V_{\sigma'}$, where $\sigma'$
is the cell of $\Delta$ with $\sigma'\not\subset S^n$ and $\sigma=
\sigma'\cap S^n$.
\end{itemize}
This completes the proof.
\end{proof}

Let $W$, $V$ be solid tori in $M$ with $c^\natural\Subset W\Subset V$.  One
can choose a Seifert fibration $\mathcal{F}$ on $M$ so that $W$ is a union
of fibers and $c^\natural$ is an exceptional fiber of order $r$.  The
restriction $\mathcal{F}_N$ of $\mathcal{F}$ on $N=M\setminus \mathrm{Int}W$ 
defines a Seifert fibration over a disk with two exceptional fibers.

Let $\Emb(W,\mathrm{Int} V)$ be the space of embeddings of $W$ into
$\mathrm{Int} V$ with the $C^\infty$-topology, and $\emb(W,\mathrm{Int} V)$
the arcwise connected component containing the inclusion $i\colon W\subset
\mathrm{Int} V$. According to Lemma~5.1 and Remark~5.2 of \cite{ga2}, 
\[\emb(W,\mathrm{Int} V)\simeq
\diff(W)\simeq \diff(\partial W)\simeq S^1\times S^1,\] where $S^1\times
S^1$ represents a free action on $\partial W$ preserving the fibration
$\mathcal{F}|_{\partial W}$. The $S^1$-action from the left factor
preserves each fiber of $\mathcal{F}|_{\partial W}$ as a set, and the one
from the right factor preserves some simple loop in $\partial W$ meeting
each fiber of $\mathcal{F}|_{\partial W}$ transversely in a single
point. The left factor action extends to a fiber-preserving $S^1$-action on
$M$, which defines a continuous map $\varphi\colon S^1\longrightarrow
\diff(M)$ with $\varphi_{y_0}=\mathrm{Id}_M$.

For any $m\in \mathbf{Z}$, we define
$\varphi^m\colon S^1\longrightarrow \diff(M)$ as follows.
\begin{itemize}
\item
$(\varphi^0)_y=\mathrm{Id}_M$ for any $y\in S^1$.
\item
For any $m>0$ (resp.\ $m<0$), $(\varphi^m)_y\colon M\longrightarrow M$ $(y\in
S^1)$ is the composition of $|m|$ copies of $\varphi_y$
(resp.\ $(\varphi_y)^{-1}$).
\end{itemize}   
Let $\mathbf{Z}_V$ be the subgroup of $\pi_1(\emb(W,\mathrm{Int}
V))$ generated by the left factor $S^1$-action.

\begin{lemma}\label{Embed}
Suppose that $f\colon S^n\longrightarrow \diff(M)$ is a continuous map
with $f_{y_0}= \mathrm{Id}_M$ and $f_y(c^\natural)\Subset V$ for any $y$ in
$S^n$.
\begin{enumerate}[\rm (i)]
\item
If $n=1$, then $f$ is homotopic rel.\ $y_0$ to $\varphi^m$ for some $m\in
\mathbf{Z}$.  Moreover, if $f$ is contractible in $\diff(M)$,
then $f$ extends to a continuous map $F\colon B^2\longrightarrow
\diff(M)$ with $F_z(c^\natural)\Subset V$ for any $z\in B^2$.
\item
If $n\neq 1$, then $f$ extends to a continuous map $F\colon
B^{n+1}\longrightarrow \diff(M)$ with $F_z(c^\natural)\Subset V$ for any
$z\in B^{n+1}$.
\end{enumerate}
\end{lemma}
\begin{proof}
Let $W$ be a solid torus with $c^\natural\Subset W\Subset V$, sufficiently
slim so that $f_y(W)\Subset V$ for any $y\in S^n$.  When $n=0$, it is not
hard to construct a homotopy $F\colon[0,1]\longrightarrow \diff(M)$ such
that $F_0=f_{y_0}$, $F_1=f_{y_1}$ and $F_t(c^\natural)\Subset V$ for any
$t\in [0,1]$, where $S^0=\{y_0,y_1\}$.  In fact, there exists an extension
$F_{[0,1/2]\cup \{1\}}$ of $f_{y_0}$ and $f_{y_1}$ with
$F_t(c^\natural)\Subset V$ for any $t\in [0,1/2]$ and $F_{1/2}|_W=F_1|_W$.
Since the Seifert fibration on $N=M\setminus \mathrm{Int} W$ has
a base orbifold with a disk as its underlying space and with two
exceptional fibers, $N$ has a unique essential annulus up to
ambient isotopy. This implies that $F_{1/2}|_{N}$ is isotopic to
$F_1|_{N}$, and consequently there is an extension $F_{[0,1]}$ of
$F_{[0,1/2]\cup \{1\}}$ with $F_{[0,1]}(c^\natural)\subset \mathrm{Int} V$.

Suppose now that $n\geq 1$.  As in the proof of \cite[p.\ 146, Claim]{ga2}
(using the Palais-Cerf covering isotopy theorem), there exists a continuous
map $K\colon S^n\times [0,1] \longrightarrow \diff(M)$ satisfying the
following conditions.
\begin{itemize}
\item
$K_{(y,0)}=f_y$ for any $y\in S^n$ and $K_{(y_0,t)}=\mathrm{Id}_M$ for any
$t\in [0,1]$.
\item
$K_{(y,t)}(W)\Subset V$ for any $(y,t)\in S^n\times [0,1]$.
\item
$K_{(y,1)}$ $(y\in S^n)$ fixes $W$ as a set.  Moreover, when $n=1$,
$K_{(y,1)}$ $(y\in S^1)$ defines an $S^1$-action on $\partial W$
preserving $\mathcal{F}|_{\partial W}$.
\end{itemize}

Consider first the case of $n=1$. If the element of
$\pi_1(\emb(W,\mathrm{Int} V))$ represented by $K_{(y,1)}$ $(y\in S^1)$
were not contained in $\mathbf{Z}_V$, then the restriction of
$K_{(y,1)}|_N$ to a basepoint $n_0\in \partial N$ for $y\in S^1$ would not
lie in the subgroup of $\pi_1(N,n_0)$ generated by a nonsingular fiber,
contradicting the fact that the restriction of a circular homotopy to any
basepoint must represent a central element of the fundamental group.
So we may choose the homotopy $K$ to satisfy
$K_{(y,1)}|_W=\varphi_y^m|_W$ $(y\in S^1)$ for some $m\in
\mathbf{Z}$. 

From Hatcher \cite{ha1}, the subspace 
of $\diff(M)$ consisting of diffeomorphisms $g$ with
$g|_W=\mathrm{Id}_M|_W$ is contractible. Since $K_{(y_0,1)}\circ
\varphi_{y_0}^{-m}=\mathrm{Id}_M$, it follows that $K_{(y,1)}\circ
(\varphi^{-m})_y$ $(y\in S^1)$ is contractible in $\diff(M)$ and hence $f$
is homotopic to $\varphi^m$ rel.\ $y_0$ in $\diff(M)$. This proves the
first part of~(i).

Assume now that $f$ is contractible, and fix a basepoint $x_0$ in $M$. The
trace homomorphism
\[\alpha\colon \pi_1(\diff(M))\longrightarrow Z(\pi_1(M))\cong \mathbf{Z}\]
is defined by putting, for any $g\colon S^1\longrightarrow \diff(M)$ with
$g_{y_0}=\mathrm{Id}_M$, $\alpha([g])$ equal to the element represented by
the loop $g_y(x_0)$ $(y\in S^1)$ in $M$. In particular, $\alpha$ maps the
class represented by $\varphi^m$ to $m\in \mathbf{Z}$. Since $f$ is
contractible, $m=0$. Regard $B^2$ as obtained from $S^1\times [0,1]$ by
shrinking $S^1\times \{1\}$ to a point.  Since
$(\varphi^0)_y=\mathrm{Id}_M$ for any $y\in S^1$, $K$ induces a continuous
map $F\colon B^2 \longrightarrow \diff(M)$ with $F|_{S^1}=f$,
$F(\mathbf{0})=\mathrm{Id}_M$ and $F_z(W) \Subset V$ for any $z\in
B^2$. This proves the remainder of~(i).

Suppose now that $n>1$. Since $\pi_n(\emb(W,\mathrm{Int} V))=\{0\}$, we
may apply the argument in part (i) to $K_{(y,1)}$ itself instead of
$K_{(y,1)}\circ (\varphi^{-m})_y$, obtaining an extension $F\colon
B^{n+1}\longrightarrow \diff(M)$ of $f$ as in~(ii).
\end{proof}

\section{Proof of the Main Theorem}\label{Proof}

As noted in Section~\ref{sec:sketch}, we may assume that $M$ is non-Haken,
and it suffices to prove that
$\pi_n(\diff(M))\cong \pi_n(S^1)$ for all $n\geq 1$. We first
examine~$n=1$.
\begin{lemma}\label{n=1}
Any continuous map $f\colon S^1\longrightarrow \diff(M)$ with
$f_{y_0}=\mathrm{Id}_M$ is homotopic to $\varphi^m$ rel.\ $y_0$ for some
$m\in \mathbf{Z}$.
\end{lemma}
\begin{proof}
Fix a cell decomposition $\Delta$ of $B^2$ and a map $V$ of $\Delta$
satisfying the conditions (i) and (ii) of Lemma \ref{Cell}.  Select a
maximal tree $T$ in $\Delta^{(1)}$ such that the complement
$\Delta^{(1)}\setminus T$ consists of elements $\sigma_1,\dots,\sigma_k$
with $y_0\in \sigma_k\subset S^1$, $S^1\setminus \sigma_k\subset |T|$ and
such that, for any $i=1,\dots,k$, there exists $\tau_i\in \Delta^{(2)}$
with $\sigma_i\subset \partial \tau_i\subset |T_i|:=|T|\cup
\sigma_1\cup\cdots\cup \sigma_i$, see Fig.\ \ref{f_tree}\,(a).
\begin{figure}[hbtp]
\centering
\includegraphics[width=0.85\textwidth]{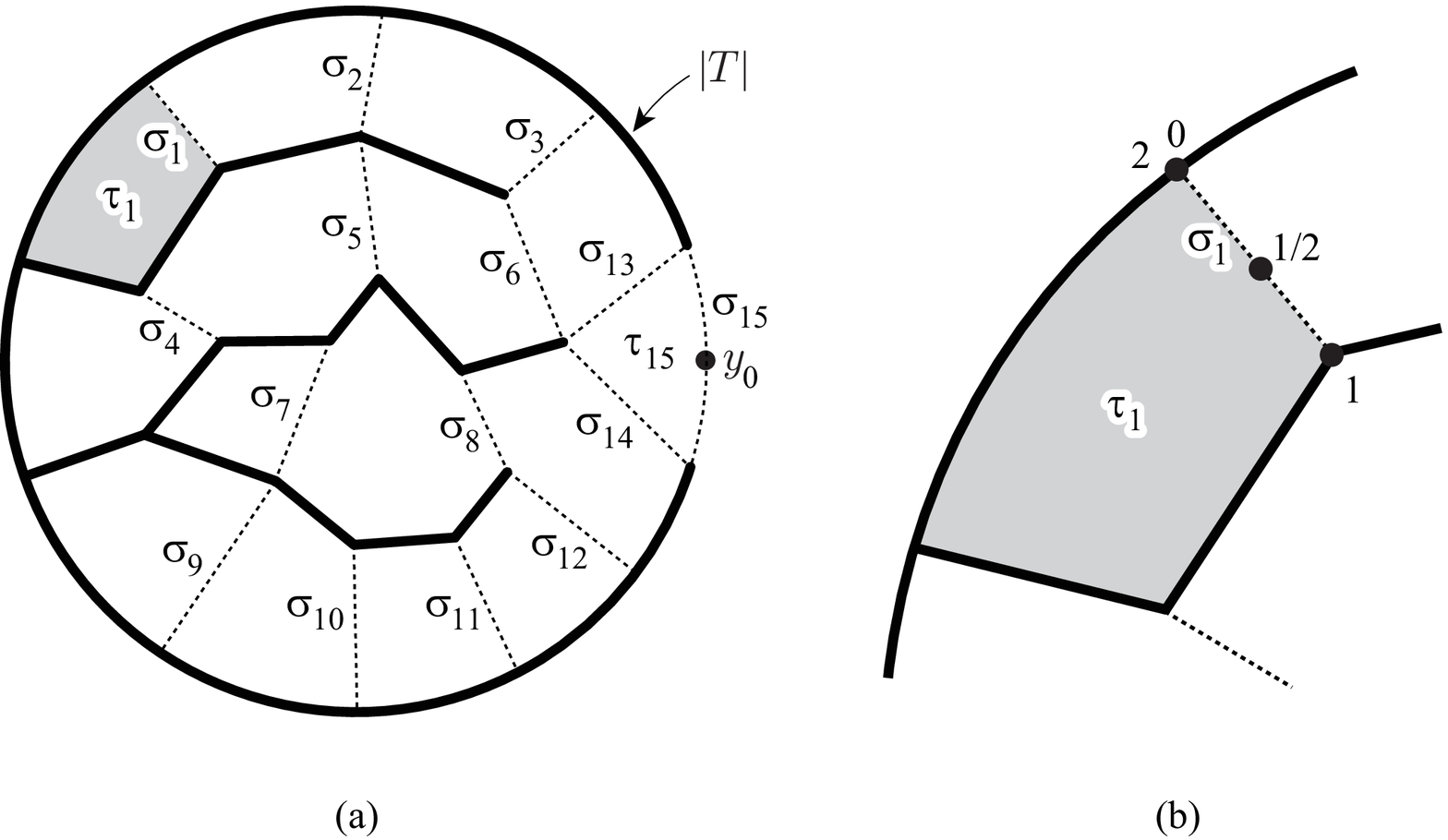}
\caption{}
\label{f_tree}
\end{figure}

For each vertex $v$ of $T|_{S^1}$, we have $f_v\in \diff(M)$ with
$f_v(c^\natural)\Subset V_v$, and for each edge $\sigma$ of $T|_{S^1}$, we
have $f_y(c^\natural)\Subset V_\sigma$ for all $y\in \sigma$. Consider an
edge $\sigma$ in $T$ having one endpoint $v$ in $S^1$ and the other
endpoint $w$ in the interior of $B^2$. Since $V_v\Subset V_\sigma$ and
$V_w\Subset V_\sigma$, we can obtain by isotopy extension a map
$F_\sigma\colon \sigma\to \diff(M)$ with $F_v=f_v$, $F_t(c^\natural)\Subset
V_\sigma$ for $t\in\sigma$, and $F_w(c^\natural)\Subset V_w$. Inducting on
the distance from $|T|\cap S^1$, we have $F_{|T|}\colon |T|\to \diff(M)$ such that
$F_v(c^\natural)\Subset V_v$ for each vertex of $T$ and
$F_t(c^\natural)\Subset V_\sigma$ for each $t$ in each edge $\sigma$ of~$T$.

Now parameterize $\sigma_1$ and $\partial \tau_1\setminus \mathrm{Int}
\sigma_1$ respectively by $[0,1]$ and $[1,2]$ so that `$0=2$' in $\partial
\tau_1$, as in Fig.\ \ref{f_tree}\,(b). We have $F_0(c^\natural)\Subset
V_{\sigma_1}$ and $F_1(c^\natural)\Subset V_{\sigma_1}$, and it follows
that there is an extension of $F_1$ to $F_{[1/2,1]}$, such that
$F_{1/2}=F_0$ and $F_t(c^\natural)\Subset V_{\sigma_1} \Subset V_{\tau_1}$
for any $t\in [1/2,1]$.

Applying Lemma \ref{Embed}\,(i) to $F_0^{-1}\circ F_t$ $(1/2\leq t\leq 2)$
and $V:= F_0^{-1}(V_{\tau_1})$, we have $j\in \mathbf{Z}$ such that the
loop product of $(\varphi^j)_{2t}$ $(t\in [0,1/2])$ and $F_0^{-1}\circ F_t$
$(t\in [1/2,2])$ is contractible in $\diff(M)$, where the domain $S^1$ of
$\varphi^j$ is supposed to be the quotient space obtained from $[0,1]$ by
identifying $0$ with $1$ and regarding the point $0\,(=1)$ as the basepoint
$y_0$ of $S^1$.  Thus the extension $F_{[0,2]}$ of $F_{[1/2,2]}$ defined by
$F_t=F_0\circ (\varphi^j)_{2t}$ $(0\leq t\leq 1/2)$ is contractible in
$\diff(M)$ and satisfies $F_t(c^\natural)\Subset V_{\sigma_1}$ for any $t\in
[0,1]$.

So far, $f_{|T|\cap S^1}$ has been extended to $F_{|T_j|}$ satisfying the following 
conditions.
\begin{enumerate}[(a)]
\item
$F_t(c^\natural)\Subset V_{\sigma_i}$ whenever $t\in \sigma_i$ 
for $i=1,\dots, j$.
\item
For any simple loop $\lambda$ in $|T_j|$, the restriction $F_\lambda$ is contractible in $\mathrm{diff}(M)$. 
\end{enumerate}
Repeating the argument, we obtain an extension
$F_{|T_{k-1}|}$ satisfying (a) and (b). Using $f$ on $\sigma_k$, we extend
$F_{|T_{k-1}|}$ to $F_{|T_{k}|}$ satisfying (a).

By the condition (b) for $j=k-1$, for any simple loop $\lambda$ in $|T_{k-1}|$, $F_\lambda$ is 
contractible.
Therefore the original $f$ is homotopic
rel.\ $y_0$ to the loop $F_{\partial \tau_k}$. Since
$F_t(c^\natural)\Subset V_{\sigma_i}\Subset V_{\tau_k}$ for each $t\in \sigma_i\subset \partial \tau_k$,
Lemma~\ref{Embed}(i) shows that 
$F_{\partial \tau_k}$ is homotopic rel.\ $y_0$ to
$\varphi^m$ for some $m\in\textbf{Z}$.
\end{proof}

\begin{proof}[Proof of the Main Theorem]
In Lemma~\ref{Embed} we defined the trace homomorphism $\alpha\colon
\pi_1(\diff(M))\longrightarrow Z(\pi_1(M))$. Lemma \ref{n=1} shows that
$\alpha$ is an isomorphism, that is, $\pi_1(\diff(M))\cong
\mathbf{Z}$. Moreover, the $S^1$-action which moves each point vertically
in its fiber defines a map $S^1\to \diff(M)$ which induces an isomorphism
on fundamental groups, so it remains to show that $\pi_n(\diff(M))=0$ for
$n> 1$.

Suppose that $n>1$ and let $f\colon S^n\longrightarrow \diff(M)$
be any continuous map with $f_{y_0}=\mathrm{Id}_M$. Let $\Delta$ be a cell
decomposition on $B^{n+1}$ and $V$ a map of $\Delta$ satisfying the
conditions of Lemma \ref{Cell}.
Let $T_0$ be a maximal subcomplex of $\Delta$ such that 
$|T_0|$ is simply connected and $S^n\subset |T_0|\subset S^n\cup |\Delta^{(1)}|$.
We set $\Delta^{(1)}\setminus T_0=\{\sigma_1,\dots,\sigma_k\}$ and 
$|T_i|=|T_0|\cup \sigma_1\cup\cdots\cup \sigma_i$ for $i=1,\dots,k$.
As in the proof of Lemma~\ref{n=1}, we can extend $f$ to $F_{|T_0|}$
satisfying the conditions (a) and (b) in the proof of Lemma \ref{n=1}.

Next we will extend $F_{|T_0|}$ to $\sigma_1$ so that
$F_{|T_1|}$ satisfies (a) and (b).
Let $v,w$ be the endpoints of $\sigma_1$.
Fix an arc $\alpha$ in $|T_0|$
from $w$ to $v$. As in the proof of Lemma~\ref{n=1}, parameterize $\sigma$
and $\alpha$ as $[0,1]$ and $[1,2]$ so that $v=0=2$, and extend
$F_{|T_0|}$ to $[1/2,1]$ so that $F_0=F_{1/2}$. Since $F_0(c^\natural)\Subset
V_v\Subset V_{\sigma_1}$, Lemma~\ref{n=1} implies that $F_{[1/2,2]}$ is
homotopic relative to $\{1/2,2\}$ to a path in $\diff(M)$ with 
$F_t(c^\natural)\Subset V_{\sigma_1}$ at each time. Using the reverse of this
path on $[0,1/2]$ gives an extension of $F_{|T_0|}$ to $F_{|T_1|}$ such that 
$F_{\sigma_1\cup \alpha}$ is a null-homotopic loop.
Since the restriction of $F_{|T_0|}$ to any loop in $|T_0|$ is contractible, 
this implies that 
$F_{\lambda_1}$ is also contractible for any loop $\lambda_1$ in $|T_1|$.

Repeating this process on $\sigma_i$ $(i=2,\dots,k)$, we obtain an
extension $F_{|T_k|}=F_{|\Delta^{(1)}|\cup S^n}$ satisfying (a) and (b).
In particular, its restriction to the boundary of any $2$-cell in $\Delta$
is null-homotopic. So Lemma~\ref{Embed}(i) implies that
$F_{|\Delta^{(1)}|\cup S^n}$ extends to $F_{|\Delta^{(2)}|\cup S^n}$,
satisfying $F_z(c^\natural)\Subset V_\tau$ for any $z$ in each $\tau\in
\Delta^{(2)}$. Then, by applying Lemma~\ref{Embed}\,(ii) repeatedly on the
higher skeleta of $\Delta$, one can extend $F_{|\Delta^{(2)}|\cup S^n}$ to
all of $|\Delta^{(n+1)}|=B^{n+1}$. It follows that $f\colon
S^n\longrightarrow \diff(M)$ is contractible and hence $\pi_n(\diff(M))=0$.
\end{proof}

\section{Deforming homotopy equivalences to diffeomorphisms}\label{sec:threecone}

The fiber-preserving diffeomorphisms of Seifert-fibered $3$-manifolds are
well-understood, see for example Section~1 of Neumann and Raymond
\cite{NR}. Apart from a few simple exceptions, Seifert fiberings of
Seifert-fibered $3$-manifolds with infinite fundamental group are unique up
to isotopy (see Lemma~2.1 and Corollary~2.3 of~\cite{oh}), and consequently
any diffeomorphism is isotopic to a fiber-preserving one.

It is also true that when $M$ is a closed Seifert-fibered $3$-manifold and
$\pi_1(M)$ is infinite, any homotopy equivalence from $M$ to $M$ is
homotopic to a diffeomorphism. This is certainly well-known in the Haken
case, by Waldhausen's celebrated results~\cite{wa2}. For the non-Haken
cases, it is folk knowledge, but we are not aware of a published proof. For
the work of this paper, we actually need only the case of hyperbolic base
orbifold, but it is appropriate to include a proof of the Euclidean base
orbifold case in order that all of our applications will also extend if our
Main Theorem can be established in the infranilmanifold case (the only
explicit invocation of the Main Theorem is in the proof of
Theorem~\ref{thm:fiber-preserving}). K.-B.~Lee has shown us a more general
proof using the theory of Seifert fiberings, but we include here an
elementary and nearly self-contained argument for the cases we need. It
requires some notational preliminaries, but they are needed for our later
work anyway.

For the remainder of this section, we assume that $M$ is Seifert-fibered
over an orbifold $O$ which is the $2$-sphere with exactly three cone
points, and that $\pi_1(M)$ is infinite. To set notation, we recall a
standard description of a Seifert-fibered structure on $M$. Remove from $O
$ the interiors of three disjoint disks, each containing one of the cone
points, to obtain a disk-with-two-holes $F$. Then $\pi_1(F)=\langle
Q_1,Q_2,Q_3\;|\; Q_1Q_2Q_3=1\rangle$, with the three boundary circles
representing the $Q_i$. Form $F\times S^1$, writing $\pi_1(S^1)=\langle
T\rangle$ and $\pi_1(F\times S^1)=\pi_1(F)\times \langle T\rangle$. To the
boundary tori, use fiber-preserving diffeomorphisms to attach suitably
Seifert-fibered solid tori, each containing an exceptional fiber, so that
the meridian curves represent $Q_i^{\alpha_i}T^{\beta_i}$, $1\leq i\leq
3$. The pairs of relatively prime integers $(\alpha_i,\beta_i)$ with
$\alpha_i\geq 2$ are called the (unnormalized) Seifert
invariants. Different choices of $\beta_i$ can yield the same (up to
orientation-preserving diffeomorphism) topological fibering, but all
choices are congruent modulo $\alpha_i$.

From the construction, we obtain the presentation
\[\pi_1(M) = \langle q_1,q_2,q_3,t\;|\;tq_it^{-1}=q_i, \;
q_i^{\alpha_i}t^{\beta_i}=1,\; 1\leq i\leq 3,\; q_1q_2q_3=1\rangle\ ,\]
where the principal fiber represents the element $t$ which generates the
center $C$ of $\pi_1(M)$. Putting $t=1$ gives the quotient
\[\pi^{\,\mathrm{orb}}_1(O) = \langle q_1,q_2,q_3\;|\;q_i^{\alpha_i}=1,\; 1\leq i\leq 3,\; q_1q_2q_3=1\rangle\ .\]

Since $M$ is aspherical, our next result implies that any homotopy
equivalence from $M$ to $M$ is homotopic to a diffeomorphism.
\begin{prop}\label{prop:strong_rigidity}
Suppose that $M$ is Seifert-fibered over an orbifold $O$ which
has three cone points and underlying manifold the $2$-sphere, and that
$\pi_1(M)$ is infinite. Let $\theta$ be an automorphism of $\pi_1(M)$. Then
there exists an orientation-preserving fiber-preserving diffeomorphism of
$M$ whose induced automorphism on $\pi_1(M)$ equals $\theta$
in~$\Out(\pi_1(M))$.
\end{prop}

\begin{proof}
Since $C$ is the center of $\pi_1(M)$, there is a commutative diagram
\[
\begin{CD}
1 @>>> C @>>> \pi @>>> \pi_1^{\,\mathrm{orb}}(O) @>>> 1\\
@VVV @V{\theta|_C}VV  @VV{\theta}V @V{\overline{\theta}}VV @VVV\\
1 @>>> C @>>> \pi @>>> \pi_1^{\,\mathrm{orb}}(O) @>>> 1\\
\end{CD}
\] 
where the vertical maps are automorphisms. Theorem~5.8.3 of \cite{ZVC},
stated in our language, says that there is an orbifold diffeomorphism
$g^{\,\mathrm{orb}}\colon O\to O$ that induces $\overline{\theta}$
on $\pi_1^{\,\mathrm{orb}}(O)$. We may assume that $g^{\,\mathrm{orb}}(F)=F$, and we write
$g\colon F\to F$ for the restriction of $g^{\,\mathrm{orb}}$.

Since $g$ is a diffeomorphism, we have
$g_\#(Q_i)=\Gamma_iQ_{\sigma(i)}^\epsilon\Gamma_i^{-1}$ for some elements
$\Gamma_i\in \pi_1(F)$, some permutation $\sigma$ of $\{1,2,3\}$, and
$\epsilon=1$ or $\epsilon=-1$ according as $g$ preserves or reverses
orientation. Since $\overline{\theta}=g^{\,\mathrm{orb}}_\#$, we can write
$\theta(q_i)=\gamma_iq_{\sigma(i)}^\epsilon\gamma_i^{-1}t^{n_i}$ for some
integers $n_i$, where $\gamma_i$ is obtained from $\Gamma_i$ by replacing
each $Q_i$ by~$q_i$.

We claim that $n_1+n_2+n_3=0$. We have in $\pi_1(F)$ that
\[1 =g_\#(Q_1Q_2Q_3)
= \Gamma_1Q_{\sigma(1)}^{\epsilon}\Gamma_1^{-1}\;
\Gamma_2Q_{\sigma(2)}^{\epsilon}\Gamma_2^{-1}\;
\Gamma_3Q_{\sigma(3)}^{\epsilon}\Gamma_3^{-1}\ .\] Since the latter word is
trivial in $\pi_1(F)$, it is freely equivalent to a product of conjugates of
$Q_1Q_2Q_3$ and $(Q_1Q_2Q_3)^{-1}$. Therefore the corresponding element
$\gamma_1q_1^{\epsilon}\gamma_1^{-1}
\gamma_2q_{\sigma(2)}^{\epsilon}\gamma_2^{-1}
\gamma_3q_{\sigma(3)}^{\epsilon}\gamma_3^{-1}$ in $\pi_1(M)$
is freely equivalent to a
product of conjugates of $q_1q_2q_3$ and $(q_1q_2q_3)^{-1}$.
Since the relation $q_1q_2q_3=1$ holds in $\pi_1(M)$, this word is trivial in
$\pi_1(M)$ and we have
\[1 =\theta(q_1q_2q_3)
= \gamma_1q_1^{\epsilon}\gamma_1^{-1}
\gamma_2q_{\sigma(2)}^{\epsilon}\gamma_2^{-1}
\gamma_3q_{\sigma(3)}^{\epsilon}\gamma_3^{-1}t^{n_1+n_2+n_3}
=t^{n_1+n_2+n_3}\ .\]
Since $C$ is infinite, this shows that $n_1+n_2+n_3=0$.

Assume for now that $\theta(t)=t$. We have
\[
t^{-\beta_i}=\theta(t^{-\beta_i})=\theta(q_i^{\alpha_i})=\gamma_i
q_{\sigma(i)}^{\epsilon\alpha_i}\gamma_i^{-1}t^{n_i\alpha_i}\ .\]
This implies that $Q_{\sigma(i)}^{\alpha_i}=1$ in $\pi_1^{\,\mathrm{orb}}(O)$, so
$\alpha_{\sigma(i)}$ divides $\alpha_i$. Since this is true for all $i$, we
have $\alpha_{\sigma(i)}=\alpha_i$. Therefore 
\[t^{-\beta_i}=
\gamma_it^{-\epsilon\beta_{\sigma(i)}} \gamma_i^{-1}t^{n_i\alpha_i}
=t^{-\epsilon\beta_{\sigma(i)}+n_i\alpha_i}\ ,\] so
$\epsilon\beta_{\sigma(i)}-\beta_i=n_i\alpha_i$. 

Suppose for contradiction that $\epsilon=-1$. Then
$\beta_{\sigma(i)}+\beta_i=-n_i\alpha_i$, and since
$\alpha_{\sigma(i)}=\alpha_i$ we have
$\beta_{\sigma(i)}/\alpha_{\sigma(i)}+\beta_i/\alpha_i=-n_i$.  Summing this
for $1\leq i\leq 3$ and using $n_1+n_2+n_3=0$ gives $\sum
\frac{\beta_i}{\alpha_i}=0$ (if we already knew that $\theta$ arose from a
fiber-preserving diffeomorphism, then this would amount to the fact that
when a Seifert-fibered $3$-manifold has an orientation-reversing
fiber-preserving diffeomorphism, the Euler number of its Seifert fibration
is $0$). If all $\alpha_i=2$, this is impossible, so we assume that
$\alpha_1\leq \alpha_2\leq \alpha_3$ with $\alpha_3\geq 3$. Since
$\beta_{\sigma(3)}/\alpha_{\sigma(3)}+\beta_3/\alpha_3$ is an integer,
$\sigma(3)\neq 3$ and we may assume that $\sigma(3)=2$ and
$\alpha_2=\alpha_3$. But then,
\[-\frac{\beta_1}{\alpha_1}=
\frac{\beta_2}{\alpha_2}+\frac{\beta_3}{\alpha_3}\]
would be an integer, a contradiction.

Let $T_1$, $T_2$, and $T_3$ be the boundary tori of $F\times S^1$, and fix
disjoint vertical annuli $A_1$ and $A_2$ connecting $T_3$ to $T_1$ and
$T_2$ respectively. Since $n_1+n_2+n_3=0$, there is a product $j$ of
fiber-preserving Dehn twists in a neighborhood of $A_1\cup A_2$ such that
$j_\#(Q_{\sigma(i)})=Q_{\sigma(i)}T^{n_i}$ for each $i$. Let $h=j\circ
(g\times 1_{S^1})$, a fiber-preserving diffeomorphism of $F\times S^1$. In
$\pi_1(F\times S^1)$ we have $h_\#(T)=T$ and
$h_\#(Q_i)=\Gamma_iQ_{\sigma(i)}\Gamma_i^{-1}T^{n_i}$. Using
$\beta_{\sigma(i)}-\beta_i=n_i\alpha_i$, we have
$h(Q_i^{\alpha_i}T^{\beta_i})
=\Gamma_iQ_{\sigma(i)}^{\alpha_{\sigma(i)}}\Gamma_i^{-1}T^{n_i\alpha_i}T^{\beta_i}
=\Gamma_iQ_{\sigma(i)}^{\alpha_{\sigma(i)}}T^{\beta_{\sigma(i)}}\Gamma_i^{-1}$. That
is, $h$ takes meridian curves in the boundaries of the fibered solid tori
of $\overline{M-F\times S^1}$ to meridian curves. Therefore $h$ extends to
a fiber-preserving diffeomorphism of $M$ inducing~$\theta$. Since
$\epsilon$ is $1$, $g$ and therefore $h$ are orientation-preserving.

Suppose now that $\theta(t)=t^{-1}$. There is an orientation-preserving
fiber-preserving diffeomorphism $\tau$ of $M$ that reverses the direction
of the fiber; on $O$ it induces a reflection through a circle
containing the three cone points, and on each of the three fibered solid
tori it is a hyperelliptic involution. Since $\tau_\#\theta(t)=t$, the
previous case gives an orientation-preserving fiber-preserving
diffeomorphism $h$ such that $\tau_\#\theta=h_\#$ and hence
$\theta=(\tau^{-1}\circ h)_\#$ in $\Out(\pi_1(M))$.
\end{proof}

The following immediate corollary
can also be proven by consideration of Euler numbers.
\begin{cor}\label{cor:threecone_op}
Suppose that $M$ is Seifert-fibered over an orbifold $O$ which
has three cone points and underlying manifold the $2$-sphere, and that
$\pi_1(M)$ is infinite. Then every diffeomorphism of $M$ is
orientation-preserving.
\end{cor}
\begin{proof}
Since $M$ is aspherical, two diffeomorphisms are homotopic if and only if
they induce the same outer automorphism of $\pi_1(M)$. By
Proposition~\ref{prop:strong_rigidity}, every homotopy class contains an
orientation-preserving diffeomorphism, and the corollary follows since $M$
is closed.
\end{proof}

\section{Isometries}\label{sec:realization}

Throughout this section we continue to assume that $M$ is Seifert-fibered
over an orbifold $O$ which is the $2$-sphere with exactly three cone
points, and that $\pi_1(M)$ is infinite. We also continue to use the
notation set up in the previous section. In this section we will analyze
the isometry groups of these~$M$.

It is known that $M$ admits an $\mathbb{H}^2\times \R$,
$\widetilde{\mathrm{SL}_2}(\R)$, $\mathrm{Nil}$, or Euclidean geometry such
that the fibers of $M$ are geodesics. Our reference for Seifert-fibered
$3$-manifolds and their geometries is~\cite{sc2}. Every isometry of $M$ is
fiber-preserving: In all cases except the Euclidean geometry, every
isometry of the universal cover $\widetilde{M}$ preserves the
$\mathbb{R}$-fibers, so this is immediate. For the Euclidean geometry, the
induced automorphism of any isometry of $M$ must preserve the center of
$\pi_1(M)$, so takes the central element $t$ represented by the principal
fiber to either $t$ or $t^{-1}$ in~$\pi_1(M)$. This implies that the lifted
isometry preserves the $\mathbb{R}$-fibers of~$\widetilde{M}$.

\begin{prop}\label{prop:realize}
Give $M$ its standard $\mathbb{H}^2\times \R$,
$\widetilde{\mathrm{SL}_2}(\R)$, $\mathrm{Nil}$, or Euclidean geometry. If
$\theta$ is any automorphism of $\pi_1(M)$, then there exists an isometry
of $M$ whose induced automorphism on $\pi_1(M)$ equals $\theta$
in~$\Out(\pi_1(M))$.
\end{prop}

\begin{proof}
From Proposition~\ref{prop:strong_rigidity}, there exists a
fiber-preserving diffeomorphism $f\colon M\to M$ with $f_\#=\theta$.

In the $\mathbb{E}^3$-case, let $\mathcal{T}(M)$ be the Teichm\"uller space
of Euclidean structures on $M$ with unit volume. For the other cases,
$\mathcal{T}(M)$ will denote the Teichm\"uller space of all geometric
structures on $M$. For $\sigma\in \mathcal{T}(M)$, let $l_\sigma$ denote
the length of a regular fiber of~$M_\sigma$.

If $M$ has an $\H\times \R$, $\mathbb{E}^3$, or $\mathrm{Nil}$ geometry,
then by \cite[Theorems 2.4, 2.6, 2.7]{oh} $\mathcal{T}(M)$ is homeomorphic
to $\R$, which corresponds to the parameter $\log(l_\sigma)$ for $\sigma\in
\mathcal{T}(M)$. (The statement of Theorem 2.4 in \cite{oh} contains a
misprint: the exponent for the closed orientable case we use
here should be $3-4\chi(X)+2k$. We remark that $\mathcal{T}(M)$ was
also found for all of these cases by R.~Kulkarni, K.-B.~Lee, and
F.~Raymond~\cite{klr} by a different method, although in the
$\mathbb{E}^3$-case $\mathcal{T}(M)$ is given there as $\R^2$ since the
volume is not normalized to be $1$.) Since $f\colon M_\sigma\to
M_{f_*(\sigma)}$ is isometric, $l_{\sigma}=l_{f_*(\sigma)}$ and hence
$\sigma=f_*(\sigma)$ in $\mathcal{T}(M)$. It follows that $f$ is isotopic
to an isometry.

If $M$ has an $\widetilde{\mathrm{SL}_2}(\R)$ geometry, then by
\cite[Theorem 2.5]{oh} (or~\cite{klr}), $\mathcal{T}(M)$ is a single
point. Again, $f$ is isotopic to an isometry.
\end{proof}

The quotient orbifold $O$ has a unique hyperbolic structure when
$\sum 1/\alpha_i<1$, and a unique Euclidean structure up to scaling when
$\sum 1/\alpha_i=1$. An isometry of $M$ induces an isometry of $O$, so the map
$\Isom(M)\to \Diff^{\,\mathrm{orb}}(O)$ taking each isometry $f$ to its induced
diffeomorphism $\overline{f}$ has image in $\Isom(O)$.

We will need some specific isometries.
\begin{lemma}\label{lem:isometries}
Give $M$ its standard $\mathbb{H}^2\times \R$,
$\widetilde{\mathrm{SL}_2}(\R)$, $\mathrm{Nil}$, or Euclidean geometry.
\begin{enumerate}
\item[(i)] There is an isometric involution
of $M$ that preserves each exceptional fiber, reverses
the direction of the fibers, and induces an orientation-reversing
reflection on $O$.
\item[(ii)] Suppose that the Seifert invariants
  $(\alpha_j,\beta_j)$ and $(\alpha_k,\allowbreak\beta_k)$ of two
  exceptional fibers of $M$ satisfy $\alpha_j=\alpha_k$ and $\beta_j\equiv
  \beta_k\bmod \alpha_j$. Then
  there is an isometric involution of $M$
  that interchanges these exceptional fibers, preserves the fiber
  direction, and on $O$ induces an orientation-preserving isometry that
  interchanges the cone points corresponding to these two exceptional
  fibers.
\end{enumerate}
\end{lemma}
\begin{proof}
For (i), consider an orientation-reversing reflection on
$O$ whose induced automorphism $\theta$ on $\pi_1^{\,\mathrm{orb}}(O)$ is
$\theta(q_1)=q_1^{-1}$, $\theta(q_2)=q_2^{-1}$, and
$\theta(q_3)=q_2q_1q_3^{-1}q_2^{-1}q_1^{-1}$. This extends to an
automorphism of $\pi_1(M)$ by putting $\theta(t)=t^{-1}$. Applying
Proposition~\ref{prop:realize} gives an isometry as in (i) inducing
$\theta$.

For part (ii), we have by assumption that $\beta_k-\beta_j=n\alpha_j$ for
some integer $n$. We proceed as in part~(i),
using an automorphism $\theta$ such that
$\theta(t)=t$, $\theta(q_j)=q_kt^n$, $\theta(q_k)=q_jt^{-n}$, and for the
remaining $q_i$, $\theta(q_i)$ is determined by the relation
$\theta(q_1q_2q_3)=1$.
\end{proof}

For $s\in \R$, let $\varphi(s)\colon M\to M$ be induced by translation by
$sL$ in the $\R$-fibers of $\widetilde{M}$, where $L$ is the length of the
principal fiber of $M$. Each $\varphi(s)=\varphi(s+1)$, so we regard
$\varphi\colon S^1\to \Isom(M)$ as a circular isotopy of $M$. These are
\textit{vertical,} that is, they take each fiber of $M$ to itself. We
denote vertical maps of $M$ by a subscript $v$, so the vertical isometries
form a subgroup $\Isom_v(M)$. Corollary~\ref{cor:threecone_op} yields
immediately
\begin{lemma}\label{lem:vertical_isometries}
No vertical diffeomorphism of $M$ can reverse the fiber direction.
Consequently, $\Isom_v(M)=S^1$.
\end{lemma}

The isometry group $\Isom(O)$ is finite of the form $C_2\times G$, where
the $C_2$-factor is generated by an orientation-reversing reflection that
fixes the cone points, and $G$ is orientation-preserving and is either
trivial, $C_2$, or $D_3$ according as the orders $\alpha_1$, $\alpha_2$,
and $\alpha_3$ of its cone points are distinct, exactly two are equal, or
all three are equal. Note that $\Isom(O)\to \Out(\pi_1^{\,\mathrm{orb}}(O))$ is
injective.

\begin{prop}\label{prop:IsomOut}
The natural map $\Isom(M)\to \Out(\pi_1(M))$ is a surjective 
homomorphism with kernel $\Isom_v(M)$. Consequently,
$\Isom(M)$ is homeomorphic to $\Out(\pi_1(M))\times S^1$.
\end{prop}
\begin{proof}
The surjectivity is from Proposition~\ref{prop:realize}. An element $f$ of
the kernel must induce the identity outer automorphism on $\pi_1^{\,\mathrm{orb}}(O)$,
so $\overline{f}$ is the identity on $O$ and therefore $f$ is vertical.
\end{proof}

\section{Fiber-preserving diffeomorphisms}

For a Seifert-fibered $3$-manifold $M$, the fiber-preserving
diffeomorphisms form a subgroup $\Diff_f(M)$ of $\Diff(M)$. From
Theorem~2.6.4 of~\cite{HKMR}, $\Diff_f(M)$ is a separable Fr\'echet
manifold, so is homotopy equivalent to a CW-complex.

Each element of $\Diff_f(M)$ induces an orbifold diffeomorphism of the base
orbifold $O$, and by Theorem~3.6.3 of~\cite{HKMR}, the map $\Diff_f(M)\to
\Diff^{\,\mathrm{orb}}(O)$ is a fibration over its image, with fiber the vertical
diffeomorphisms~$\Diff_v(M)$. 

We will need a description of the connected component of the identity,
$\diff_v(M)$. Provided that $M$ has an
orientable base orbifold, it has a circular vertical isotopy that rotates
each nonsingular fiber one full turn, such as the
$\varphi$ in the special case of Section~\ref{sec:realization}.
\begin{lemma}\label{lem:diffv}
Let $M$ be an orientable Seifert-fibered $3$-manifold with orientable base
orbifold. Any circular vertical isotopy $\varphi\colon S^1\to \diff_v(M)$
that rotates each nonsingular fiber one full turn defines a homotopy
equivalence $S^1\simeq \diff_v(M)$.
\end{lemma}
\begin{proof}
Fix a basepoint $m_0$ in a nonsingular fiber. Restriction to $m_0$ defines
a map (actually a fibration) $e\colon \diff_v(M)\to S^1$. The composition
$S^1\stackrel{\varphi}{\longrightarrow}
\diff_v(M)\stackrel{e}{\longrightarrow} S^1$ is a homeomorphism, so
$\varphi_\#\colon \pi_1(S^1)\to \pi_1(\diff_v(M))$ is injective.

Now, consider a parameterized family $f\colon (S^q,s_0)\to
(\diff_v(M),\mathrm{Id}_M)$, for $q\geq 1$. To complete the proof that
$\varphi$ is a homotopy equivalence, we show that $f$ is nullhomotopic,
when $q>1$, or homotopic to a power of $\varphi$, when $q=1$. Multiplying
$f$ by a power of $\varphi$, when $q=1$, we may assume that
$S^q\stackrel{f}{\longrightarrow}
\diff_v(M)\stackrel{e}{\longrightarrow}S^1$ is nullhomotopic.

Let $F$ be the surface obtained from the base orbifold by removing the
interiors of disjoint disk neighborhoods of the cone points, or if there
are no cone points, by removing the interior of some disk. Consider the
restriction of $f$ to a parameterized family $g\colon S^q\longrightarrow
\diff_v(F\times S^1)$ of vertical diffeomorphisms of $F\times S^1$. Since
$S^q\stackrel{f}{\longrightarrow}
\diff_v(M)\stackrel{e}{\longrightarrow}S^1$ is nullhomotopic, we can lift
$g$ to $\tilde{g}\colon S^q\longrightarrow \diff_v(F\times \mathbb{R})$
such that $\tilde{g}(s_0)=\mathrm{Id}_{F\times \mathbb{R}}$.  Note that
for any $s\in S^q$, $\tilde{g}(s)$ is equivariant.  This means that if we
write $\tilde{g}(s)(x,t)=(x,\omega_s(x,t))$ for $(x,t)\in F\times
\mathbb{R}$ and regard $S^1$ as $\mathbb{R}/\mathbb{Z}$, then
$\omega_s(x,t+1)=\omega_s(x,t)+1$. The homotopy $\tilde{g}_u\colon
S^q\longrightarrow \diff_v(F\times \mathbb{R})$ $(u\in [0,1])$ defined by
\[\tilde{g}_u(s)(x,t)=(x,(1-u)\omega_s(x,t)+ut)\] satisfies $\tilde{g}_0(s)
=\tilde{g}(s)$, $\tilde{g}_1(s)=\mathrm{Id}_{F\times
\mathbb{R}}$ for any $s\in S^q$ and $\tilde{g}_u(s_0)
=\mathrm{Id}_{F\times \mathbb{R}}$ for any $u\in [0,1]$.
Moreover, from the construction of $\tilde{g}_u$, each
$\tilde{g}_u(s)$ is equivariant.  Thus $\tilde{g}_u$ covers a
homotopy $g_u\colon S^q\longrightarrow \diff_v(F\times S^1)$ between
$g$ and $\mathrm{Id}_{F\times S^1}$, which is naturally extended to a
homotopy $f_u\colon S^q\longrightarrow \diff_v(M)$ between $f$ and
$\mathrm{Id}_M$. This shows that $f$ is contractible in $\diff_v(M)$.
\end{proof}
\noindent We remark that when $M$ has nonorientable base orbifold, there is
no circular isotopy such as $\varphi$, and $\diff_v(M)$ is contractible, but
we will not need this information.

\begin{lemma}\label{lem:diffv3cone}
Suppose that $M$ is a Seifert-fibered $3$-manifold with base
orbifold a $2$-sphere with three cone points, and that $\pi_1(M)$ is
infinite. Then $\diff_v(M)=\Diff_v(M)$.\par
\end{lemma}
\begin{proof}
We must show that any vertical diffeomorphism $j$ of $M$ is vertically
isotopic to the identity. By Lemma~\ref{lem:vertical_isometries}, $j$
cannot reverse the fiber direction. 
By vertical isotopy, we can make $j$
the identity on an exceptional fiber $F_0$, and then the identity on a
fibered solid torus neighborhood $V$ of $F_0$. In $N=M\setminus
\text{int}(V)$, there is a vertical annulus that separates $N$ into two
solid tori $T_1$ and $T_2$, each intersecting $V$ in a vertical annulus.
By a vertical isotopy fixing $V$, we can make $j$ the identity on $T_1$. It
is now the identity on $\partial T_2$, so by vertical isotopy we can make
it the identity on $T_2$ as well.
\end{proof}

\begin{prop}\label{prop:threearrows}
Suppose that $M$ is a Seifert-fibered $3$-manifold with base orbifold
a $2$-sphere with three cone points, and that $\pi_1(M)$ is
infinite. Give $M$ its standard $\mathbb{H}^2\times \R$,
$\widetilde{\mathrm{SL}_2}(\R)$, $\mathrm{Nil}$, or Euclidean geometry.
In the sequence
\[\Isom(M)\to \Diff_f(M)\to \Diff(M)\to \Out(\pi_1(M))\ ,\]
each of the three maps is bijective on path components.
\end{prop}
\begin{proof}
By Proposition~\ref{prop:realize}, the composition of all four maps is
surjective, hence so is $\Diff(M)\to \Out(\pi_1(M))$. By results of
Scott~\cite{sc3} and Boileau-Otal~\cite{bo},
any diffeomorphism of $M$ that is homotopic to the identity is isotopic to
the identity, so $\Diff(M)\to \Out(\pi_1(M))$ is injective on path
components. This proves the lemma for the third map, and that the
second map is surjective on path components.

As usual, let $F$ be the surface obtained from the base orbifold by
removing the interiors of disjoint disk neighborhoods of the cone
points. Consider a fiber-preserving diffeomorphism $f$ of $M$ that is
isotopic to the identity. By fundamental group considerations, $f$ cannot
reverse the direction of the fiber, and must preserve each exceptional
fiber. So by fiber-preserving isotopy we may assume that $f$ is the
identity on $\overline{M-F\times S^1}$. Every orientation-preserving
diffeomorphism of $F$ that preserves each boundary component is isotopic to
the identity, allowing us to change $f$ to be the identity in the
$F$-coordinate of $F\times S^1$. Since $f$ is now vertical,
Lemma~\ref{lem:diffv3cone} shows that $f$ is vertically isotopic to
the identity. We conclude that the second map is bijective and the first
map is surjective on path components.

By Proposition~\ref{prop:IsomOut}, $\Isom(M)\to \Out(\pi_1(M))$ is
injective on path components, hence so is the first map. This completes
the proof.
\end{proof}

\section{The space of Seifert fiberings and the Smale Conjecture}
\label{sec:Seifert}

Let $M$ be a Seifert-fibered $3$-manifold. Two (smooth) Seifert fiberings
of $M$ are considered equivalent if there is a diffeomorphism of $M$ that
takes fibers of one to fibers of the other. The coset space
$\Diff(M)/\Diff_f(M)$ is the space of Seifert fiberings equivalent to the
given one. Since fiberings equivalent under $\Diff(M)$ produce conjugate
subgroups for $\Diff_f(M)$, the homeomorphism type of $\Diff(M)/\Diff_f(M)$
is independent of the particular fibering within its equivalence class.
Taking the disjoint union of copies of $\Diff(M)/\Diff_f(M)$, one for each
equivalence class of Seifert fibering, we obtain the space $\SF(M)$ of
Seifert fiberings of $M$. By Proposition~3.6.11 of~\cite{HKMR}, $\SF(M)$ is
a separable Fr\'echet manifold locally modeled on an infinite-dimensional
separable Fr\'echet space, and consequently it has the homotopy type of a
CW-complex.
\longpage

In this section, we will prove that when $M$ is a closed orientable Seifert
fibered $3$-manifold with a hyperbolic base $2$-orbifold, $\SF(M)$ is
contractible. If in addition the base orbifold is a $2$-sphere with three
cone points, and $M$ has its standard $\H^2\times\R$ or
$\widetilde{\mathrm{SL}_2}(\R)$ geometry, then the inclusion $\Isom(M)\to
\Diff(M)$ is a homotopy equivalence. Both of these facts rely upon the
following result:
\begin{theorem}\label{thm:fiber-preserving}
Let $M$ be a closed orientable Seifert-fibered $3$-manifold with a
hyperbolic base orbifold. Then the inclusion $\Diff_f(M)\to \Diff(M)$
is a homotopy equivalence.
\end{theorem}

\begin{proof}
When $M$ is Haken, this is Theorem~3.8.1 of~\cite{HKMR}, so we need only
consider the case when the base orbifold is a $2$-sphere with three cone
points. By Proposition~\ref{prop:threearrows}, the inclusion is a bijection
on path components, so it remains to prove that $\diff_f(M)\to \diff(M)$ is
a homotopy equivalence.

By Theorem~3.6.3 of~\cite{HKMR}, the induced map $\Diff_f(M)\to
\Diff^{\,\mathrm{orb}}(O)$ is a fibration over its image, and consequently the
restriction $\diff_f(M)\to \diff^{\,\mathrm{orb}}(O)$ is a fibration.  The fiber is
$\Diff_v(M)\cap \diff_f(M)$, which must be $\diff_v(M)$ by
Lemma~\ref{lem:diffv3cone}. Moreover, $\diff^{\,\mathrm{orb}}(O)$ is
contractible, since it is essentially $\diff(S^2\setminus\{\text{three
points}\})$, and it follows that the inclusion $\diff_v(M)\to \diff_f(M)$
is a homotopy equivalence.

Consider the composition $S^1\stackrel{\varphi}{\longrightarrow}
\diff_v(M)\to \diff_f(M)\to \diff(M)$. The first map is the homotopy
equivalence of Lemma~\ref{lem:diffv}, and we have just seen that the
second map is a homotopy equivalence. By the Main Theorem, the entire
composition is a homotopy equivalence, hence so the third map.
\end{proof}

The quotient map $\Diff(M)\to \Diff(M)/\Diff_f(M)$
is a fibration, by Proposition~3.6.11 of~\cite{HKMR}.
Therefore Theorem~\ref{thm:fiber-preserving} yields
\begin{cor}\label{coro:Seifert}
Let $M$ be a closed orientable Seifert-fibered $3$-manifold 
with a hyperbolic base orbifold. Then $\SF(M)$ is contractible.
\end{cor}

As another consequence of Theorem~\ref{thm:fiber-preserving}, we have
the Smale Conjecture for our class of non-Haken manifolds:
\begin{theorem}\label{thm:Smale}
Let $M$ be a closed orientable Seifert-fibered $3$-manifold having an
$\H^2\times\R$ or $\widetilde{\mathrm{SL}_2}(\R)$ geometry, and
base orbifold a $2$-sphere with three cone
points. Then the inclusion $\Isom(M)\to \Diff(M)$ is a homotopy
equivalence.
\end{theorem}

\begin{proof}
By Theorem~\ref{thm:fiber-preserving}, it suffices to show that
the inclusion $\Isom(M)\to \Diff_f(M)$ is a homotopy equivalence.

As already noted, Theorem~3.6.3 of~\cite{HKMR} shows that the induced map
$\Diff_f(M)\to \Diff^{\,\mathrm{orb}}(O)$ is a fibration over its image, which we will
denote by $\Diff^{\,\mathrm{orb}}_0(O)$. This gives the second row of the diagram
\begin{equation*}
\begin{CD}
\Isom_v(M) @>>> \Isom(M) @>>> \Isom_0(O)\\
\alpha @VVV @VVV \beta @VVV\\
\Diff_v(M) @>>> \Diff_f(M) @>>> \Diff^{\,\mathrm{orb}}_0(O)
\end{CD}
\end{equation*}
\noindent 
In the first row, $\Isom_0(O)$ is the image of $\Isom(M)\to \Isom(O)$. The
second map is a homomorphism with kernel $\Isom_v(M)$, so the first row is
also a fibration. The inclusion $\alpha$ is a homotopy equivalence by
Lemmas~\ref{lem:vertical_isometries},
\ref{lem:diffv}, and~\ref{lem:diffv3cone}.

We claim that the inclusion $\beta$ is also a homotopy equivalence, from
which it follows that the middle vertical map is as well. Each non-identity
element of $\Isom_0(O)$ is nonisotopic to the identity on
$\diff(S^2\setminus\{\text{three points}\})$, so $\beta$ is injective on
path components. Let $f\in \Diff_f(M)$ induce $\overline{f}$ on $O$. By
Proposition~\ref{prop:threearrows}, $f$ is isotopic through
fiber-preserving diffeomorphisms to an isometry, so $\overline{f}$ is
orbifold-isotopic to an isometry of $O$. That is, $\beta$ is surjective on
path components. Finally, the components of $\Diff^{\,\mathrm{orb}}(O)$ are
contractible, and the components of $\Isom_0(O)$ are points, so $\beta$ is
a homotopy equivalence and the proof is complete.
\end{proof}

\end{document}